\documentclass[reqno]{amsart}

\usepackage{amsmath}
\usepackage{amssymb}
\usepackage[margin=1in]{geometry}
\usepackage{xcolor}
\usepackage{comment}
\usepackage{graphicx}
\usepackage{empheq}

\theoremstyle{plain}
\newtheorem{thm}{Theorem}[section]
\newtheorem{lem}{Lemma}[section]
\newtheorem{prop}{Proposition}[section]
\newtheorem{cor}{Corollary}[section]

\theoremstyle{definition}
\newtheorem{defn}{Definition}[section]

\theoremstyle{remark}

\newtheorem{rmk}{Remark}[section]

\newcommand{\tr}[1]{\operatorname{tr}\left(#1\right)}
\newcommand{\diag}[1]{\left[#1\right]}
\newcommand{\sgn}[1]{\operatorname{sgn}(#1)}
\newcommand{\R}{\mathbb{R}}
\newcommand{\N}{\mathbb{N}}
\newcommand{\F}{\mathcal{F}}
\usepackage{cite}
\usepackage{hyperref}
\hypersetup{ colorlinks = true, urlcolor = blue, linkcolor = blue, citecolor = red }
\usepackage{enumitem}
\usepackage{nameref}
\makeatletter
\let\orgdescriptionlabel\descriptionlabel
\renewcommand*{\descriptionlabel}[1]{%
	\let\orglabel\label
	\let\label\@gobble
	\phantomsection
	\edef\@currentlabel{#1\unskip}%
	\let\label\orglabel
	\orgdescriptionlabel{#1}%
}
\numberwithin{equation}{section}

\begin{document}
	
\title[ Harnack inequality for anisotropic equations with nonstandard growth]{ Harnack inequality for anisotropic fully nonlinear equations with nonstandard growth}
	
\author[Byun]{Sun-Sig Byun}
\address{Department of Mathematical Sciences and Research Institute of Mathematics,
	Seoul National University, Seoul 08826, Republic of Korea}
\email{byun@snu.ac.kr}
	
\author[Kim]{Hongsoo Kim*}
\address{Department of Mathematical Sciences, Seoul National University, Seoul 08826, Republic of Korea}
\email{rlaghdtn98@snu.ac.kr}
	
\thanks {S.-S. Byun was supported by Mid-Career Bridging
	Program through Seoul National University. H. Kim was supported by the National Research Foundation of Korea(NRF) grant funded by the Korea government [Grant No. 2022R1A2C1009312].}
	
	\makeatletter
	\@namedef{subjclassname@2020}{\textup{2020} Mathematics Subject Classification}
	\makeatother
	
\everymath{\displaystyle}
	
\begin{abstract}
	We establish Harnack inequalities for viscosity solutions of a class of degenerate fully nonlinear anisotropic elliptic equations with nonstandard growth, including anisotropic-$(p_i)$-Laplacian type equations with rough coefficients.
    The estimates hold under the explicit exponent condition:
    \begin{align*}
        p_n <  \frac{\left(n- \frac{\lambda}{\Lambda}\right)\bar{p}}{n-\frac{\lambda}{\Lambda}\bar{p}}\quad \text{ or }\quad \bar{p} \geq\frac{\Lambda}{\lambda}n,
    \end{align*}
    where $\bar{p}$ is the harmonic mean of $2 \leq p_1 \leq \cdots \leq p_n$.
    The key ingredient is a new max-type anisotropic barrier, constructed by maximizing over an auxiliary variable and balancing the different directional homogeneities.
    Combined with measure estimates and intrinsic-geometry arguments, this barrier yields the Harnack inequality.
    
    In the constant-coefficient case $\lambda=\Lambda$, we prove that the exponent condition is sharp for both weak and viscosity solutions.
    Using the same max-type approach, in the supercritical regime we construct a nontrivial nonnegative entire solution vanishing on a hyperplane, while at the critical threshold we construct positive entire solutions with arbitrarily large Harnack ratios.
    Finally, we exhibit a fully nonlinear counterexample showing that, for the general class, dependence of the admissible exponent range on the ellipticity ratio $\frac{\lambda}{\Lambda}$ cannot in general be removed.
\end{abstract}
	
\maketitle
	
\section{Introduction} \label{sec1}
In this paper, we establish Harnack inequalities for viscosity solutions of the following class of inequalities in nondivergence form:
\begin{align} \label{PDE}
	\begin{cases}
		\mathcal{M}^-_{\lambda,\Lambda} \left( \diag{|D_iu|^{\frac{p_i-2}{2}}} D^2u \diag{|D_iu|^{\frac{p_i-2}{2}}} \right) -\mu\sum_i|D_iu|^{p_i-1} \leq c_0, \\
		\mathcal{M}^+_{\lambda,\Lambda} \left( \diag{|D_iu|^{\frac{p_i-2}{2}}} D^2u \diag{|D_iu|^{\frac{p_i-2}{2}}} \right)  +\mu\sum_i|D_iu|^{p_i-1} \geq -c_0,
	\end{cases}
	\quad \text{in } \Omega,
\end{align}
where $2 \leq p_1 \leq \cdots \leq p_n$, $0<\lambda \leq \Lambda$, $c_0 \geq0$, $\mu\geq0$, and $\diag{|z_i|^{\frac{p_i-2}{2}}} =\operatorname{diag}(|z_i|^{\frac{p_i-2}{2}})$ is the diagonal matrix with entries $|z_i|^{\frac{p_i-2}{2}}$ on the diagonal.

This structure encompasses several examples of equations including:
	\begin{enumerate}
		\item The degenerate anisotropic-$(p_i)$-Laplacian equation.
		\begin{align} \label{piLap}
			\Delta_{(p_i)}u(x)=  \sum_i \frac{1}{p_i-1}\partial_i(|\partial_iu|^{p_i-2}\partial_{i}u)=\sum_i |\partial_iu|^{p_i-2}\partial_{ii}u = f(x).  
        \end{align}
		\item More generally, let $A(x)=(a_{ij}(x))$ be symmetric with eigenvalues in $[\lambda,\Lambda]$ and $|b_i(x)|\leq\mu$, 
		\begin{align} \label{roughLap}
			\sum_{i,j} a_{ij}(x)|\partial_iu|^{\frac{p_i-2}{2}}|\partial_ju|^{\frac{p_j-2}{2}} \partial_{ij}u + \sum_i b_i(x)|\partial_iu|^{p_i-2}\partial_iu= f(x).    \end{align}
	\end{enumerate}
The degenerate anisotropic-$(p_i)$-Laplacian \eqref{piLap} has attracted considerable attention and has been the subject of extensive research in recent decades.
This interest stems from two key features that make its regularity theory particularly delicate.

The first is its anisotropic degeneracy, which is more degenerate than that of the standard $p$-Laplacian $\Delta_{p} u = \sum_i \partial_i(|\nabla u|^{p-2} \partial_iu)$.
While the degeneracy of the $p$-Laplacian occurs where $Du=0$.
In contrast, the anisotropic-$(p_i)$-Laplacian degenerates whenever one component of the gradient, $D_iu$, vanishes for some coordinate direction $i$ with $p_i>2$. 
Hence, even when the exponents $p_i$ are all equal, several regularity properties had remained open until recently.
Lipschitz regularity for weak solutions was established in
\cite{Bousquet18}, while the authors of \cite{Demengel162, Demengel16} established it for viscosity solutions. 
Moreover, while basic estimates such as the Harnack inequality for weak solutions can be obtained by the classical De Giorgi-Nash-Moser theory for $p$-growth, the corresponding result for viscosity solutions was  recently proved in \cite{Kim26}.
For additional related results, see \cite{Bousquet23,Bousquet16}.

The second feature is that the equation exhibits nonstandard growth, which means that the equation arises from the Euler-Lagrange equation of a functional of the following type:
\begin{align*}
    w \mapsto \int F(Dw) dx, \quad \text{ where } \quad |z|^{p_1} \lesssim F(z) \lesssim |z|^{p_1}+|z|^{p_n}.
\end{align*}
Functionals of this type have been studied intensively since the foundational works \cite{Marcellini89,Marcellini91} and they now constitute a central topic in the regularity theory of nonlinear elliptic equations.
See \cite{Marcellini20,Mingione21} for broad surveys of the field and \cite{Kim252} for related results in the viscosity setting.

A key issue in establishing regularity for problems with nonstandard growth is the smallness of the gap between $p_1$ and $p_n$. 
Without such a condition,  counterexamples to regularity are known; see \cite{Giaquinta87,Marcellini89}.
The sharp condition for local boundedness of weak solutions of \eqref{piLap} established in \cite{Fusco90,Boccardo90} is 
\begin{align} \label{optpcond}
    p_n \leq  \frac{n\overline{p}}{n-\overline{p}}, \quad \text{ where } \overline{p} = \left( \frac{1}{n}\sum_i\frac{1}{p_i} \right)^{-1}.
\end{align}
See also \cite{Cupini15,Cupini17, Fusco93, Cianchi00,Cupini09} for further results on local boundedness.
Moreover, Lipschitz regularity was recently established in \cite{Bousquet20} for bounded weak solutions of \eqref{piLap} without any assumption on the closeness of $p_1$ and $p_n$, and the analogous result for viscosity solutions was proved in \cite{Kim252}. 
See also \cite{Demengel17,Bousquet24}.

In contrast to local boundedness, the optimal exponent range for the Harnack inequality has remained unclear.
To the best of our knowledge, even for the constant-coefficient anisotropic-$(p_i)$-Laplacian, the sharp range for the validity of Harnack estimates has not been identified.
For equations with rough coefficients, the situation is even less understood.

In the rough-coefficient setting in the degenerate regime, the available regularity theory mainly concerns H\"older continuity and is known only under additional restrictions on the exponents.
For weak solutions, H\"older regularity has been established in \cite{Liskevich09,Marcellini14} and \cite{Piro-Vernier19} for special cases of the exponents $(p_i)$, satisfying 
\begin{align*}
    p_1<p_2 =\cdots =p_n \quad \text{ or }\quad p_1=\cdots =p_{n-1}<p_n,
\end{align*}
under the boundedness condition \eqref{optpcond}.
Moreover, when the exponents $p_i$ can be all distinct, the authors of \cite{Dibenedetto16} established H\"older regularity, under the small-gap condition
\begin{align} \label{pwcond}
    p_n - p_1 \leq \frac{1}{q} \quad \text{ for a sufficiently large } q=q(n,(p_i),\lambda,\Lambda).
\end{align}
Related developments can also be found in \cite{Liao20,Ciani25,Feo21,Baldelli24, Ciani232}.
These works mainly rely on De Giorgi-type techniques and the intrinsic geometry framework introduced in \cite{DiBenedetto08,DiBenedetto12} to address the inhomogeneous scaling of the equations.

In the constant-coefficient case, the Harnack inequality was established in \cite{Ciani23} under the condition
\begin{align} \label{ppcond}
     2<p_1, \quad p_n < \overline{p}\left( 1+\frac{1}{n}\right),
\end{align}
which coincides with the exponent range arising in the corresponding parabolic theory.
Their proof is based on a self-similar Barenblatt solution, constructed by abstract functional-analytic methods and used as a barrier.
This result provides an important sufficient range for the Harnack inequality, but does not identify the sharp elliptic threshold.

Against this background, essentially no corresponding Harnack theory is available for viscosity solutions of the fully nonlinear class \eqref{PDE}.
The methods used for weak solutions do not directly extend to this setting.
In particular, the weak formulations and energy estimates underlying De Giorgi-type arguments are unavailable for the general viscosity class considered here.
Despite these technical challenges, we obtain an explicit admissible exponent range for the Harnack estimates in the class \eqref{PDE}, including the anisotropic-$(p_i)$-Laplacian with rough coefficients \eqref{roughLap}.
For the constant-coefficient anisotropic-$(p_i)$-Laplacian case, we further prove that this range is sharp, including the critical endpoint, for both viscosity and weak solutions.

We now state the main theorem.
	
\begin{thm} \label{Mainthm}
	Let $u \in C(\Omega)$ be nonnegative and satisfy \eqref{PDE} in the viscosity sense in $\Omega$. 
    Assume that $2\leq p_1\leq \cdots \leq p_n$ and that
    \begin{align} \label{pcond}
        p_n <  \frac{\left(n- \frac{\lambda}{\Lambda}\right)\bar{p}}{n-\frac{\lambda}{\Lambda}\bar{p}}\quad \text{ or }\quad \bar{p} \geq\frac{\Lambda}{\lambda}n, \quad \text{ where } \quad \overline{p} = \left( \frac{1}{n}\sum_i \frac{1}{p_i}\right)^{-1}.
    \end{align}
	Then there exist constants $C_0,R_0>1$ and $\epsilon_1,\epsilon_2 >0$ depending on $n,(p_i),\lambda$ and $\Lambda$ such that if $u(0)>0$,  the following statements hold:
    \begin{enumerate}
        \item For any $r \in(0,1]$, if $K_{R_0r}(u(0)) \subset \Omega$, $c_0 \in [0,\epsilon_1 u(0)^{p_n-1}]$ and $\mu \in [0,\mu_1]$ where $\mu_1=\mu_1(u(0),n,(p_i)$, $\lambda,\Lambda)>0$, then
	\begin{align*}
		 \sup_{K_r(u(0))}u\leq C_0u(0). 
	\end{align*}
        \item For any $r \in(0,1]$, if $K_{R_0r}(u(0)/C_0) \subset \Omega$, $c_0 \in [0,\epsilon_2u(0)^{p_n-1}]$ and $\mu \in [0,\mu_2]$ where $\mu_2=\mu_2(u(0),n,(p_i),\lambda,\Lambda)>0$, then
        \begin{align*}
		  \inf_{K_r(u(0)/C_0)}u \geq \frac{1}{C_0}u(0) . 
	    \end{align*}
    \end{enumerate}
    Here,
    \begin{align*}
        K_r(M) :=\{|x_i| < r^{\alpha_i}M^{\beta_i}\} \quad \text{ with } \quad  \alpha_i = \frac{1}{p_i}, \quad \beta_i = -\left( \frac{p_n-p_i}{p_i} \right).
    \end{align*}
\end{thm}

A direct consequence of the main theorem is the following H\"older regularity.
\begin{cor} \label{Holdercor}
    Let $u \in C(\Omega)$ be a viscosity solution of \eqref{PDE}, and assume that $(p_i)$ satisfy \eqref{pcond}.
    Then $u$ is locally H\"older continuous in $\Omega$, with a universal H\"older exponent  $\alpha=\alpha(n,(p_i),\lambda,\Lambda) \in (0,1)$.
\end{cor}
\begin{rmk}
    The smallness assumptions on $c_0$ and $\mu$ in Theorem \ref{Mainthm} do not impose an additional local restriction.
    Indeed, by the intrinsic scaling described in Remark \ref{scal} both quantities can be made sufficiently small by restricting to a sufficiently small scale $r$.
    Furthermore, when $c_0=\mu=0$, no such restriction on the radius is needed, and the estimates in Theorem \ref{Mainthm} hold for arbitrary $r>0$, provided the corresponding intrinsic cubes are contained in $\Omega$.
\end{rmk}

When $\lambda=\Lambda$ and $c_0=\mu=0$, the inequalities \eqref{PDE} reduce to the constant-coefficient anisotropic-$(p_i)$-Laplacian equation
\begin{align*}
    \Delta_{(p_i)}u=0.
\end{align*}
In this case, the exponent condition \eqref{pcond} in Theorem \ref{Mainthm} becomes
\begin{align} \label{pcondLap}
    p_n  < \frac{\left(n- 1\right)\bar{p}}{n-\bar{p}}, \quad\text{or} \quad \bar{p}\geq n.
\end{align}
The following theorem shows that this range is sharp, including the critical endpoint.
\begin{thm} \label{Mainthm2}
    Let $n \geq 2$, $2\leq p_1\leq \cdots \leq p_n$ and assume $\bar{p} <n$.
    \begin{enumerate}
        \item If 
        \begin{align} \label{psupcrit}
            p_n  > \frac{\left(n- 1\right)\bar{p}}{n-\bar{p}},
        \end{align}
        then there exists a nontrivial nonnegative viscosity solution $W \in C^1(\R^n)$ of
        \begin{align*}
            \Delta_{(p_i)}W(x)=0 \quad \text{ in } \R^n,
        \end{align*}
        such that
        \begin{align*}
            W =0 \quad \text{ on } \{x_n = 0\}, \quad W >0 \text{ on } \{x_n \neq 0\}.
        \end{align*}
        \item If
        \begin{align} \label{pcrit}
            p_n  = \frac{\left(n- 1\right)\bar{p}}{n-\bar{p}},
        \end{align}
        then there exists a family $\{W_\epsilon\}_{\epsilon>0} \subset C^1(\R^n)$ of positive viscosity solutions
        \begin{align*}
            \Delta_{(p_i)}W_\epsilon(x)=0 \quad \text{ in } \R^n,
        \end{align*}
        satisfying
        \begin{align*}
            \frac{\sup_{K_1(W_\epsilon(0))}W_\epsilon}{W_\epsilon(0)} \longrightarrow \infty,
        \end{align*}
        as $\epsilon\rightarrow0$.
    \end{enumerate}
    Consequently, the exponent condition \eqref{pcondLap} is sharp for the constant-coefficient anisotropic-$(p_i)$-Laplacian.
\end{thm}

\begin{cor}
    Let $2\leq p_1\leq \cdots \leq p_n$.
    Then for the constant-coefficient anisotropic-$(p_i)$-Laplacian, the exponent condition \eqref{pcondLap} is sharp for the Harnack inequality for both viscosity and weak solutions.
\end{cor}
\begin{proof}
    The viscosity case follows from Theorems \ref{Mainthm} and \ref{Mainthm2}.
    For weak solutions, note that \eqref{pcondLap} is strictly stronger than the sharp local boundedness condition \eqref{optpcond}.
    Hence every local weak solution in this range is locally bounded, and therefore locally Lipschitz by \cite{Bousquet20}.
    Since weak and viscosity solutions coincide in the constant-coefficient anisotropic-$(p_i)$-Laplacian setting (see \cite{Demengel17, Ochoa24}), Theorem \ref{Mainthm} yields the corresponding Harnack inequality for weak solutions.
    Conversely, the $C^1$ counterexamples constructed in Theorem \ref{Mainthm2} are also weak solutions.
    Therefore, the same threshold is sharp for both viscosity and weak solutions.
\end{proof}

\subsection{Discussion of the exponent condition}
We first comment on the significance of the harmonic mean $\bar p$ in the exponent condition.
In the constant-coefficient case, the sharp condition \eqref{pcondLap} has the same harmonic mean structure as the optimal local boundedness condition \eqref{optpcond}, with $n$ replaced by $n-1$ in the numerator.
In particular, the Harnack inequality holds whenever $\bar p>n$.
This parallels the anisotropic Sobolev--Morrey embedding in the variational theory, which yields H\"older continuity under the same condition.
It is noteworthy that the same harmonic-mean structure emerges here from a purely viscosity-based argument, without using energy estimates or Sobolev inequalities.

For the fully nonlinear class \eqref{PDE}, the exponent condition also involves the ellipticity ratio  $\lambda/\Lambda$.
Indeed, the counterexample in Proposition \ref{lamprop} shows that the sharp constant-coefficient range cannot be retained uniformly for arbitrary ellipticity ratios.
Thus, some dependence on $\lambda/\Lambda$ is unavoidable in the general fully nonlinear setting.
We do not, however, claim that the precise dependence appearing in \eqref{pcond} is optimal.

Finally, although we restrict ourselves to the degenerate regime $p_i\ge2$, the computations underlying the proof appear to extend naturally to the mixed singular--degenerate regime in which $1<p_i<2$ for some $i$.
The main obstruction is not the exponent calculation itself, but rather the viscosity formulation: when $p_i<2$, the operator becomes singular at points where $D_i u=0$, and the viscosity framework used in this paper is no longer directly applicable.
We therefore leave this regime for future investigation.

\subsection{Outline of the proof}
We now describe the main ideas and the novel aspects of the proof.
Due to the non-homogeneous scaling of \eqref{PDE}, the argument is carried out in the intrinsic geometry introduced in \cite{DiBenedetto08,DiBenedetto12}.
In those works, this approach was used to establish the Harnack inequality for the parabolic $p$-Laplacian.
A recent extension to viscosity solutions of the corresponding parabolic problem was achieved in \cite{Vedansh25}.
The core idea is that, when controlling the superlevel set $\{u>M\}$, we use the intrinsic cube $K_r(M)$ instead of the standard cube, which is compatible with the natural scaling of the equation at level $M$.
A major difficulty is that, when $M$ is large, $K_r(M)$ becomes very flat in most directions, which severely complicates any covering or iteration argument used in the proof. 
Unlike the parabolic $p$-Laplacian, where the intrinsic geometry distinguishes only the space and time variables, the present geometry may scale differently in every coordinate direction.
Despite these difficulties, by suitably adapting a technique developed in \cite{Vedansh25}, we are able to derive a Krylov--Safonov type regularity result.

The first ingredient is the basic measure estimate in Lemma \ref{basicmeaslem}.
Following the sliding paraboloid method of Cabr\'e \cite{Cabre97} and Savin \cite{Savin07}, we slide test functions from below and estimate the measure of the resulting contact set through the area formula.
To reflect the coordinate-wise degeneracy of the equation, we replace the usual quadratic paraboloids by anisotropic test functions of the form
 \begin{align*}
        \varphi(x) =  \sum_i K_i |x_i|^{\frac{p_i}{p_i-1}},
    \end{align*}
following the approach introduced in \cite{Kim26,Byun25}.

The main new ingredient is the construction of an anisotropic barrier for the doubling property.
Because the operator has different homogeneities in different coordinate directions, it is difficult to construct a suitable barrier directly as a function of $x$.
We therefore introduce an auxiliary scalar variable $t$, which may be viewed as representing the size of the barrier, and consider
\begin{align*}
    F(x,t) = \sum_i d_i |x_i|^{\frac{p_i}{p_i-1}}t^{\frac{q}{p_i-1}}.
\end{align*}
The powers are chosen so that
\begin{align*}
    |D_iF|^{p_i-2}D_{ii}F \simeq t ^{q},
\end{align*}
and hence the directional terms are comparable to the same quantity $t^q$, independently of $i$.
We then eliminate this auxiliary parameter by maximizing over $t$, which leads to the max-type barrier
\begin{align*}
     \Phi(x) :=  \max_{t\geq0}\left\{ t-F(x,t)\right\} = \max_{t\geq0}\left\{ t-\sum_i d_i |x_i|^{\frac{p_i}{p_i-1}}t^{\frac{q}{p_i-1}}\right\}.
\end{align*}
If $T(x)$ denotes the maximizing parameter, then $T(x)$ is comparable to $\Phi(x)$, and the maximizing condition allows the differential structure built into $F$ to pass naturally to $\Phi$.
By appropriately selecting parameters $d_i >0$ and $q>p_n-1$, this yields the required barrier under the exponent condition \eqref{pcond}.
Exploiting the scaling property of the barrier function and applying the comparison principle with the solution, we then deduce the desired doubling property in Lemma \ref{doublem}.

Combining the basic measure estimate and the doubling property with a covering argument for the intrinsic cubes $K_r(M)$, we obtain the $L^\epsilon$ estimates in Theorem \ref{Lepsilonthm}, which give algebraic decay of the superlevel sets $\{u>M\}$.
The Harnack inequality in Theorem \ref{Mainthm} then follows by the usual iteration in the intrinsic geometry.

We finally describe the counterexamples establishing the sharpness of the exponent condition.
The same max-type construction used for the barrier also suggests the counterexamples in the constant-coefficient case.
Above the critical threshold, the strict supercritical inequality allows us to modify the max-type construction and obtain a nontrivial nonnegative entire solution that vanishes on a hyperplane.
At the critical threshold, we replace the corresponding power term by its logarithmic endpoint analogue and introduce a small parameter $\epsilon>0$, obtaining a family of strictly positive entire solutions whose Harnack ratios diverge as $\epsilon\to0$.
These two constructions prove the sharpness of the constant-coefficient exponent range.
We also provide a fully nonlinear counterexample showing that this sharp constant-coefficient range cannot be retained uniformly for arbitrary ellipticity ratios.

The paper is organized as follows.
In Section \ref{sec2}, we introduce notation and some preliminaries.
In Section \ref{sec3}, we prove the basic measure estimate, Lemma \ref{basicmeaslem}.
In Section \ref{sec4}, we construct an anisotropic barrier function and in Section \ref{sec5}, we prove the doubling property using the barrier function.
In Section \ref{sec6}, we establish the $L^\epsilon$ estimates for supersolutions, and we prove the main theorem, Theorem \ref{Mainthm}, in Section \ref{sec7}.
In section \ref{sec8},
we prove the sharpness result in Theorem \ref{Mainthm2} and construct a counterexample showing that dependence on the ellipticity ratio $\lambda/\Lambda$ is unavoidable for the general fully nonlinear class.

\section{Notation and Preliminaries} \label{sec2}
Throughout this paper, we write a point $x \in \mathbb{R}^n$ as $x=(x_1,\cdots ,x_n)$. 
We denote by $S(n)$ the space of symmetric $n \times n$ real matrices, and $I_n \in S(n)$ denotes the identity matrix.
For $r>0$ and $x_0 \in \R^n$, the standard cube is defined as $Q_r(x_0) = \{ x \in \mathbb{R}^n : |x_i-(x_0)_i| <r \text{ for any } i\in\{ 1, \cdots,n\}\}$.
We also write $Q_r = Q_r(0)$.

Moreover, we define the intrinsic cube adapted to the anisotropic scaling of the equation as
\begin{align} \label{intcubedef}
    aK_r(M,x_0) :=  \{x \in \R^n : |x_i-(x_0)_i| < ar^{\alpha_i}M^{\beta_i}\} \quad \text{ where } \quad  \alpha_i = \frac{1}{p_i}, \ \beta_i = -\left( \frac{p_n-p_i}{p_i} \right),
\end{align}
for $a,r,M>0$ and $x_0 \in \R^n$.
We abbreviate $K_r(M,x_0)=1 \cdot K_r(M,x_0)$ and $K_r(M)=K_r(M,0)$.
Observe that $\alpha_1 \geq \cdots  \geq \alpha_n >0$ and $\beta_1 \leq  \cdots \leq \beta_n=0$.
Thus, for $0<N\leq M$ and $0<r \leq s$, we have $K_r(M) \subset K_s(N)$.

For $a=(a_i) \in \mathbb{R}^n$, we always write the diagonal matrix with entries $a_i$ as
	\begin{align*}
		\diag{a_i} := \operatorname{diag}(a_1,\cdots,a_n) \in S(n).
	\end{align*}

We recall the definition and basic properties of the Pucci extremal operators (see \cite{Caffarelli95}).
\begin{defn}
	For given $0<\lambda \leq \Lambda$, the Pucci operators $\mathcal{M}_{\lambda,\Lambda}^{\pm} : S(n) \rightarrow \mathbb{R}$ are defined as follows:
	\begin{align*}
		\mathcal{M}_{\lambda,\Lambda}^{+}(M) := \lambda \sum_{e_i(M)<0}e_i(M) + \Lambda \sum_{e_i(M)>0}e_i(M), \quad 
		\mathcal{M}_{\lambda,\Lambda}^{-}(M) := \Lambda \sum_{e_i(M)<0}e_i(M) + \lambda \sum_{e_i(M)>0}e_i(M),
	\end{align*}
	where $e_i(M)$ are the eigenvalues of $M$.
	We also abbreviate $\mathcal{M}^{\pm}_{\lambda,\Lambda}$ as $\mathcal{M}^{\pm}$.
	\end{defn}
    \begin{prop} \label{pucciprop}
		For any $M,N \in S(n)$, we have
		\begin{enumerate}
            \item $\mathcal{M}^{\pm}(cM) = c\mathcal{M}^{\pm}(M)$ for $c>0$.
			\item $\mathcal{M}^{-}(-M) = -\mathcal{M}^{+}(M) $. 
            \item For a nonnegative definite $M\geq 0$, $\mathcal{M}^{-}_{\lambda,\Lambda}(M) = \lambda\tr{M} $ and  $\mathcal{M}^{+}_{\lambda,\Lambda}(M) = \Lambda\tr{M} $.
			\item $ \mathcal{M}^{-}(M) +\mathcal{M}^{-}(N)\leq \mathcal{M}^{-}(M+N) \leq \mathcal{M}^{-}(M) +\mathcal{M}^{+}(N).$ 
		\end{enumerate}
	\end{prop}
We also recall the definition of the inequalities \eqref{PDE} in the viscosity sense, see \cite{Ishii92, Caffarelli95}.
\begin{defn}
	We say that $u \in C(\Omega)$ satisfies 
	$$\mathcal{M}^\pm_{\lambda,\Lambda} \left( \diag{|D_iu|^{\frac{p_i-2}{2}}} D^2u \diag{|D_iu|^{\frac{p_i-2}{2}}} \right) \pm\mu\sum_i|D_iu|^{p_i-1} \leq c_0,  \quad\text{in}\ \Omega \quad (\text{resp.} \geq-c_0) $$
	in the viscosity sense, if for any $x_0 \in \Omega$ and any test function $\psi \in C^2(\Omega)$ such that $u-\psi$ has a local minimum (resp. maximum) at $x_0$, then
	$$\mathcal{M}^\pm_{\lambda,\Lambda} \left( \diag{|D_i\psi(x_0)|^{\frac{p_i-2}{2}}} D^2\psi(x_0) \diag{|D_i\psi(x_0)|^{\frac{p_i-2}{2}}} \right) \pm\mu\sum_i|D_i\psi(x_0)|^{p_i-1} \leq c_0  \quad (\text{resp.} \geq-c_0).$$ 
    Throughout the paper, when $p_i=2$ we interpret
$|q_i|^{(p_i-2)/2}$ as $1$, including at $q_i=0$.
\end{defn}
Since the equation \eqref{PDE} is nonhomogeneous, we need to employ the following intrinsic scaling framework.
\begin{rmk}[Scaling] \label{scal}
    For $r,M>0$, if $u \in C(K_r(M))$ satisfies
    \begin{align} \label{scalrmkpde}
        \mathcal{M}^- \left( \diag{|D_iu|^{\frac{p_i-2}{2}}} D^2u \diag{|D_iu|^{\frac{p_i-2}{2}}} \right) -\mu\sum_i|D_iu|^{p_i-1} \leq c_0 \quad \text{ in }K_r(M),
    \end{align}
    then the function $v \in C(K_1(1))$ defined as
    \begin{align*}
        v(x) = \frac{u(r^{\alpha_i}M^{\beta_i} x_i)}{M} \quad \text{ with } \quad  \alpha_i = \frac{1}{p_i}, \quad \beta_i = -\left( \frac{p_n-p_i}{p_i} \right),
    \end{align*}
    satisfies the following inequality
    \begin{align*}
        \mathcal{M}^- \left( \diag{|D_iv|^{\frac{p_i-2}{2}}} D^2v \diag{|D_iv|^{\frac{p_i-2}{2}}} \right) -\mu\sum_i \frac{r^{1/p_i}}{M^{(p_n-p_i)/p_i}}|D_iv|^{p_i-1} \leq \frac{r}{M^{p_n-1}}c_0 \quad \text{ in }K_1(1).
    \end{align*}
    In particular, when $r\leq1$ and $M\geq1$, $v$ also satisfies the same inequality \eqref{scalrmkpde} in $K_1(1)$.
    Moreover, if $c_0=\mu=0$, then $v$ satisfies the same inequality for any $r,M>0$.
    By choosing $r>0$ sufficiently small, we can assume that $c_0$ and $\mu$ are arbitrarily small.
    
\end{rmk}
We recall a comparison principle for viscosity solutions, see \cite{Ishii92}.
\begin{lem}[Comparison principle] \label{comp}
    Let $F=F(r,p,X) \in C(\R \times \R^n \times S(n))$ be degenerate elliptic and strictly decreasing in $r$.
    More precisely, assume:
    \begin{enumerate}
        \item $F(r,p,X) \leq F(r,p,Y)$ for any $X \leq Y$.
        \item There exists $\gamma>0$ such that $\gamma(r-s) \leq F(s,p,X)- F(r,p,X)$ for any $s \leq r$. 
    \end{enumerate}
    Then $F$ satisfies the comparison principle: if $\Omega\subset \R^n$ is a bounded domain, and $u \in USC(\overline{\Omega})$ and $v \in LSC(\overline{\Omega})$ are viscosity subsolutions and supersolutions, respectively, such that $F(v,Dv,D^2v) \leq 0 \leq F(u,Du,D^2u)$ in $\Omega$, and $u \leq v$ on $\partial \Omega$, then $u \leq v$ in $\Omega$.
\end{lem}
We conclude this section by stating the Vitali covering lemma for intrinsic cubes.
\begin{lem} \label{Vitalilem}
    Fix $M>0$ and let $\mathcal{F} = \{K_i=K_{r_i}(M,x_i)\}$ be a collection of intrinsic cubes with  $x_i \in \R^n$ and $r_i \in (0,1]$.
    Then there exists a countable subcollection $\mathcal{G} = \{K_j\} \subset \mathcal{F}$ of pairwise disjoint cubes such that
    \begin{align*}
        \bigcup_{K_i \in \mathcal{G}} 5K_i \supset \bigcup_{K_i \in \mathcal{F}} K_i.
    \end{align*}
\end{lem}
\begin{proof}
We may assume that $\mathcal{F}$ is countable since $\R^n$ is a Lindel\"of space.
For $k \geq 0$, we define the partition of $\mathcal{F}$ as
\begin{align*}
    \F_k = \{ K_i \in \F : 2^{-k-1} < r_i \leq 2^{-k} \}.
\end{align*}
We now construct the desired subfamily $\mathcal{G}$ inductively.
We define $\tilde{\F}_0 \subset\F_0$ as a maximal family of disjoint cubes in $\F_0$.
We also define $\tilde{\F}_k\subset\F_k$ for $k>0$ inductively.
Assume that $\tilde{\F}_j$ have been defined for $0 \leq j \leq k-1$.
Then for a subfamily
\begin{align*}
    \mathcal{H}_k := \{ K_i \in \F_k : K_i \cap K = \emptyset \quad \text{ for any } K \in \bigcup_{j=0}^{k-1} \tilde{\F}_j \},
\end{align*}
we define $\tilde{\F}_k$ as a maximal family of disjoint cubes in $\mathcal{H}_k$.
Then we set $\mathcal{G} = \bigcup_{k=0}^{\infty}\tilde{\F}_k$.
By construction, the cubes in $  \mathcal{G}$ are pairwise disjoint.
To verify the covering property, for $K_l \in \F_k$ with some $k \geq 0$, we claim that $K_l  \subset \bigcup_{j=0}^{k} \bigcup_{K_m \in \tilde{\F}_j} 5K_m$.
Note that if $K_l \notin \tilde{\F}_k $, then there exists $K_m \in  \bigcup_{j=0}^{k} \tilde{\F}_j$ satisfying $K_l \cap K_m \neq \emptyset$ by maximality of $\tilde{\F}_k$.
Since $K_l \in \F_k $ implies $2^{-k-1} < r_l \leq 2^{-k}$ and $K_m \in  \bigcup_{j=0}^{k} \tilde{\F}_j$ implies $ 2^{-k-1}\leq r_m \leq 1$, we obtain $r_l \leq 2 r_m$.
We prove the claim by showing that $K_l \subset 5K_m$. For $x \in K_l$ and $y \in K_l \cap K_m$, we have
\begin{align*}
    |(x-x_m)_i| &\leq |(x-y)_i| +|(y-x_m)_i| 
    \leq 2r_l^{\alpha_i}M^{\beta_i} +r_m^{\alpha_i}M^{\beta_i} \leq 5r_m^{\alpha_i}M^{\beta_i}. 
\end{align*}
Thus, $x \in 5K_m$, which implies $K_l \subset 5K_m$.
This completes the proof.
\end{proof}
	
\section{Basic measure estimate} \label{sec3}
In this section we prove a basic measure estimate via the sliding-paraboloid method.
We adapt this approach by replacing standard quadratic paraboloids with a class of anisotropic test functions \eqref{anipara}, which is designed to match the anisotropic degeneracy of the operator.
However, the function defined in \eqref{anipara} fails to be $C^2$ on the coordinate hyperplanes $\bigcup_{\{i:p_i\neq2\}}\{x_i=0\}$.
To overcome this technical issue, we rely on the `restriction to slices' argument introduced in \cite{Kim26}, which reduces the analysis of the singular set to lower-dimensional coordinate slices where the test function remains smooth.

Note that the assumption \eqref{pcond} on $(p_i)$ is not required for the validity of the following lemma.

\begin{lem} \label{basicmeaslem}
    For any $\epsilon_0 >0$, there exist $r_0  \in (0,1)$ and $M_0>1$ depending only on $n,(p_i),\lambda,\Lambda$,  and $\delta_0  \in (0,1)$ depending additionally on $\epsilon_0$ such that
    if $u \in C\left(\overline{K_1(M_0)}\right)$ is nonnegative in $K_1(M_0)$ and satisfies
    \begin{align*}
		\mathcal{M}^- \left( \diag{|D_iu|^{\frac{p_i-2}{2}}} D^2u \diag{|D_iu|^{\frac{p_i-2}{2}}} \right) -\Lambda\sum|D_iu|^{p_i-1} \leq  \epsilon_0 \quad \text{ in } K_1(M_0),	
	\end{align*}
    with 
    \begin{align*}
        \inf_{K_{r_0}(1)}u \leq 1,
    \end{align*}
    then we have $$|\{u <M_0\} \cap K_1(M_0)\}| \geq \delta_0|K_1(M_0)|.$$
\end{lem}
\begin{proof}
    Using the inf convolution of $u$,
	we may assume that $u$ is semiconcave and twice differentiable almost everywhere (see \cite{Caffarelli95, Imbert16}). 
    Given that $\inf_{K_{r_0}(1)}u \leq 1$, there exists a point $x_0 \in \overline{K_{r_0}(1)}$ such that $u(x_0) \leq 1$.
    We now apply the sliding paraboloid method using the following `anisotropic paraboloid' defined as
    \begin{align} \label{anipara}
        \varphi(x) = -\sum_i K^{\frac{1}{p_i-1}} (p_i-1)^{\frac{1}{p_i-1}} \frac{p_i-1}{p_i} |x_i|^{\frac{p_i}{p_i-1}},
    \end{align}
    where $K \geq 1$ is a constant to be determined.
    Observe that $\frac{1}{2}\leq(p_i-1)^{\frac{1}{p_i-1}} \frac{p_i-1}{p_i}\leq 2 $.
    We define $V=K_{r_0}(1)$ as the vertex set, with $r_0<1$ to be determined.
    We slide the paraboloid with a vertex in $V$ from below until it touches the graph of $u$.
    The touching point set is then defined as
    \begin{align*}
		T=\{x\in\overline{K_{1}(M_0)}: \exists y \in V \text{ such that } u(x) - \varphi(x-y) = \min_{z\in \overline{K_{1}(M_0)}} u(z)-\varphi(z-y)  \}.
	\end{align*}
    For $y \in V$ and $z \in \partial K_1(M_0)$, we have 
    \begin{align*}
        u(z)-\varphi(z-y) &\geq \frac{1}{2}\min_i \{K^{\frac{1}{p_i-1}}  |M_0^{-\frac{p_n-p_i}{p_i}}-r_0^{\frac{1}{p_i}}|^{\frac{p_i}{p_i-1}}\} \\
        &\geq \frac{1}{4}\min_i \{K^{\frac{1}{p_i-1}}  M_0^{-\frac{p_n-p_i}{p_i-1}}\} = \frac{1}{4} M_0,
    \end{align*}
    by choosing $K=M_0^{p_n-1}$ and $r_0 \leq M_0^{-(p_n-p_1)}A^{-p_n}$ with a sufficiently large constant $A \geq 2$, which implies $|M_0^{-\frac{p_n-p_i}{p_i}}-r_0^{\frac{1}{p_i}}| \geq \frac{1}{2}M_0^{-\frac{p_n-p_i}{p_i}}$.
    Conversely, for $y \in V$ and $x_0 \in \overline{K_{r_0}(1)}$, we have
    \begin{align*}
        u(x_0)-\varphi(x_0-y) &\leq 1+2\sum_i K^{\frac{1}{p_i-1}} |2r_0^{\frac{1}{p_i}}|^{\frac{p_i}{p_i-1}} \\
        &\leq 1+8 \sum_i(Kr_0)^{\frac{1}{p_i-1}} \\
        &\leq 1+ 8 \sum_i (M_0^{p_1-1}A^{-p_n})^{\frac{1}{p_i-1}} \\
        &\leq 1+8n M_0A^{-\frac{p_n}{p_n-1}} < \frac{1}{4}M_0,
    \end{align*}
    by choosing large $A > 64n$ and $M_0 >8$.
    These estimates imply that the minimum must be attained at an interior point.
    Therefore, we have $T \subset \{  u <M_0\} \cap K_1(M_0)$.

    For each $x\in T$ there is a unique vertex point $y\in V$ such that $\varphi(z-y)+C$ touches $u$ from below at $x$, so we define the mapping $m :T\rightarrow V$ as $m(x)=y$.
    We also have
    \begin{align} \label{grad}
		D_iu(x)&= D_i\varphi(x-y) = -(K(p_i-1))^{\frac{1}{p_i-1}}|x_i-y_i|^{\frac{1}{p_i-1}} \sgn{x_i-y_i}, \\
		D^2u(x) &\geq D^2\varphi(x-y) = -\diag{K^{\frac{1}{p_i-1}}(p_i-1)^{-\frac{p_i-2}{p_i-1}}|x_i-y_i|^{-\frac{p_i-2}{p_i-1}}}. \label{hess}
	\end{align}
    Note that $D^2\varphi(x-y)$ is not well-defined whenever $p_i\neq2$ and $x_i=y_i$ for some coordinate $i$.
    Using \eqref{grad}, the map $m$ can be written explicitly as
    \begin{align} \label{mequ}
        y_i = x_i + \frac{1}{K(p_i-1)}|D_iu|^{p_i-2}D_iu(x).
    \end{align}
    We now decompose the touching point set $T$ according to whether each coordinate direction is degenerate or non-degenerate.
    Let $I_2 :=\{i:p_i=2\}$. 
    For $ I_2\subset I \subset \{1,\cdots,n\} $ and $\epsilon>0$, we define
		\begin{align}  \label{defT}
			\begin{split}
				T^{\epsilon}_{I} &= \{ x \in T : |x_i-y_i| >\epsilon \ \text{ for }\ i\in I \setminus I_2, \ \ |x_j-y_j|=0 \ \text{ for }\ j\notin I\},\\
				T_{I} &= \{ x \in T : |x_i-y_i| >0 \ \text{ for }\ i\in I\setminus I_2, \ \ |x_j-y_j|=0 \ \text{ for }\ j\notin I\},
			\end{split}
		\end{align}
		where $y=m(x)$ and define $V^{\epsilon}_{I} = m(T^{\epsilon}_{I})$ and $V_{I}=m(T_I)$.
		Then for $0<\epsilon_1 < \epsilon_2$, we have
		\begin{align} \label{TV}
			\begin{split}
				T^{\epsilon_2}_I \subset T^{\epsilon_1}_I, \quad T_I = \bigcup_{\epsilon>0} T^\epsilon_I, \quad T=\bigsqcup_{I_2 \subset I \subset \{1,\cdots,n\}} T_I,\\
				V^{\epsilon_2}_I \subset V^{\epsilon_1}_I, \quad V_I = \bigcup_{\epsilon>0} V^\epsilon_I, \quad V=\bigcup_{I_2 \subset I \subset \{1,\cdots,n\}} V_I.
			\end{split}
		\end{align}
    Our aim is to prove $|V_I| \leq C|T_I|$ for each $I_2 \subset I \subset \{1,\cdots,n\}$ by applying the area formula.

    \textbf{ Case 1 : $I = \{1,\cdots,n\}$.}
	For $x \in T_I^\epsilon$ and $y=m(x)$, there exists a paraboloid $\varphi(z-y)+C_0$ that touches $u$ from below at $z=x$.
	Moreover, since $|x_i-y_i| >\epsilon$ for every $i \in \{1,\cdots,n\}\setminus I_2$, $\varphi(z-y)$ is smooth at $z=x$ with $|D^2\varphi(x-y)| \leq C\epsilon^{-2}$.
	Therefore, for any $z \in K_1(M_0)$, we get
	\begin{align}
		u(z) &\geq \varphi(z-y) +C_0 \nonumber \\
		&\geq \varphi(x-y) + C_0+ \langle D\varphi(x-y),z-x\rangle - C\epsilon^{-2}|z-x|^2 \nonumber \\
		&=u(x) + \langle Du(x),z-x\rangle - C\epsilon^{-2}|z-x|^2, \label{touchbelow}
	\end{align}
	by the Taylor expansion of $\varphi(\cdot-y)$ at $x$, the fact that $u(x)=\varphi(x-y)+C_0$, and \eqref{grad}.
		
	Let $x_1, x_2 \in T_I^\epsilon $ and set $y_1 =m(x_1)$ and $ y_2=m(x_2)$.
	For any $z \in B_r(x_1)$ with $r=2|x_1-x_2|$, we have
	\begin{align*}
		u(z) & \geq u(x_1) + \langle Du(x_1),z-x_1 \rangle - C\epsilon^{-2}|z-x_1|^2 \\
		& \geq u(x_2) + \langle Du(x_2),x_1-x_2 \rangle - C\epsilon^{-2}|x_1-x_2|^2 + \langle Du(x_1),z-x_1 \rangle - C\epsilon^{-2}|z-x_1|^2,
	\end{align*}
	by applying \eqref{touchbelow} at $x_1$ and $x_2$.
	By the semiconcavity of $u$, there exists $\tilde{C}>0$ such that
	\begin{align*}
		u(z) \leq u(x_2) + \langle Du(x_2),z-x_2 \rangle +\tilde{C}|z-x_2|^2.
	\end{align*}
	Using the two inequalities above, we obtain
	\begin{align*}
		\langle Du(x_1)-Du(x_2),z-x_1 \rangle \leq (C\epsilon^{-2}+\tilde{C})r^2.
	\end{align*}
	Choosing $z$ such that $z-x_1$ is parallel to $Du(x_1)-Du(x_2)$, we obtain
	\begin{align*}
		|Du(x_1)-Du(x_2)| \leq (C\epsilon^{-2}+\tilde{C})|x_1-x_2|.
	\end{align*}
	Therefore, we conclude that $Du$ is Lipschitz in $T_I^\epsilon$, and the area formula can be applied to the map $m$.
    By differentiating \eqref{mequ} with respect to $x$, we obtain
    \begin{align*}
        D_xy= I_n + \frac{1}{K}\diag{|D_iu|^{p_i-2}}D^2u.
    \end{align*}
    By the identity $\det (I_n+MN)=\det(I_n+NM)$ for $n\times n$ matrices $M$ and $N$, we have
    \begin{align*}
        \det D_xy&= \det \left(I_n + \frac{1}{K}\diag{|D_iu|^{p_i-2}}D^2u \right) \\
        &= \det \left(I_n + \frac{1}{K}\diag{|D_iu|^{\frac{p_i-2}{2}}}D^2u \diag{|D_iu|^{\frac{p_i-2}{2}}} \right)=: \det B.
    \end{align*}
    A direct calculation using \eqref{grad} and \eqref{hess} shows that
    \begin{align*}
        -\frac{1}{K}\diag{|D_i\varphi|^{\frac{p_i-2}{2}}}D^2\varphi \diag{|D_i\varphi|^{\frac{p_i-2}{2}}} =I_n ,
    \end{align*}
    which implies
    \begin{align*}
        B &=I_n + \frac{1}{K}\diag{|D_iu|^{\frac{p_i-2}{2}}}D^2u \diag{|D_iu|^{\frac{p_i-2}{2}}} \\
        & =\frac{1}{K}\diag{|D_iu|^{\frac{p_i-2}{2}}}(D^2u-D^2\varphi) \diag{|D_iu|^{\frac{p_i-2}{2}}} \geq0.
    \end{align*}
    Equation \eqref{grad} also gives $ |D_iu|^{p_i-1} \leq 2K(p_i-1)$.
    Therefore,
    \begin{align*}
        \lambda \tr{B} &= \mathcal{M}^-(B) \\
        & \leq \mathcal{M}^{+}(I_n) + \frac{1}{K}\mathcal{M}^{-}\left(\diag{|D_iu|^{\frac{p_i-2}{2}}}D^2u \diag{|D_iu|^{\frac{p_i-2}{2}}}\right) \\
        &\leq n\Lambda + \frac{1}{K}(\epsilon_0 + \Lambda\sum|D_iu|^{p_i-1}) \\
        &\leq C(1+\epsilon_0),
    \end{align*}
    where $C$ depends only on $n,(p_i),\lambda$ and $\Lambda$.
    Consequently,
    \begin{align*}
        \det D_xy = \det B \leq \left( \frac{\tr{B}}{n} \right)^n \leq C(1+\epsilon_0)^n.
    \end{align*}
    Applying the area formula to $m$, we obtain
    \begin{align*}
        |V_I^\epsilon| \leq \int_{T^\epsilon_I} |\det D_xy| \ dx  \leq \int_{T^\epsilon_I} C(1+\epsilon_0)^n \, dx \leq C(1+\epsilon_0)^n |T^\epsilon_I|.
    \end{align*}
    Letting $\epsilon \rightarrow0$ yields $|V_I| = \lim_{\epsilon\rightarrow0}|V_I^\epsilon| \leq \lim_{\epsilon\rightarrow0}C(1+\epsilon_0)^n |T^\epsilon_I| \leq C(1+\epsilon_0)^n |T_I|.$

    \textbf{ Case 2 : $I_2\subset I \subsetneq \{1,\cdots,n\}$ and $I \neq \emptyset$.} 
	Let $k=|I|$ and $J = \{1,\cdots,n\} \setminus I$.
    For $x \in T^\epsilon_I$, we have $D_ju(x)=0$ for $j \in J$ by \eqref{mequ}.
	Let $(a_j)_{j\in J}$ be $(n-k)$ numbers satisfying $a_j \in [-M_0^{\beta_j},M_0^{\beta_j}]$.
	We define the $k$-dimensional slice 
	$$S(a_j) = \{z:z_j=a_j \text{ for  } {j \in J}\} \cap K_1(M_0),$$
    and set $T^{\epsilon}_{I}(a_j) = S(a_j) \cap T^\epsilon_I $ and $V^{\epsilon}_{I}(a_j) =S(a_j) \cap V^\epsilon_I$.
    Let $u_{(a_j)}, \varphi_{(a_j)}$ and $m_{(a_j)}$ be the restriction of $u,\varphi,m$ to $S(a_j)$, respectively.
    Note that for $x \in T^\epsilon_I(a_j)$, we have $y_j=x_j=a_j$ for $y=m(x)$ and $j \in J$, which implies $y \in V^{\epsilon}_{I}(a_j)$.
    Therefore, the restricted map $m_{(a_j)} : T^{\epsilon}_{I}(a_j) \rightarrow V^{\epsilon}_{I}(a_j)$ is well defined with
		\begin{align} \label{mequ2}
			y_i =x_i + \frac{1}{K(p_i-1)}|D_iu_{(a_j)}|^{p_i-2}D_iu_{(a_j)}(x).
		\end{align}
    Moreover, since $\varphi_{(a_j)}(z-y)+C_0$ touches $u_{(a_j)}$ from below at $z=x$, we also obtain
		\begin{align} \label{gradhessequ}
			Du_{(a_j)}(x)= D\varphi_{(a_j)}(x-y) \in \mathbb{R}^{k}, \quad D^2u_{(a_j)}(x) \geq D^2\varphi_{(a_j)}(x-y)\in S(k).
		\end{align}
    Observe that since $|x_i-y_i|>\epsilon$ for any $i\in I\setminus I_2$, $\varphi_{(a_j)}(z-y)$ is smooth at $z=x$ and $D^2\varphi_{(a_j)}(x-y)$ is well defined with $|D^2\varphi_{(a_j)}(x-y)| \leq C\epsilon^{-2}$.
    Moreover, $u_{(a_j)}$ is also semiconcave.
    Therefore, by the same argument as in \textbf{Case 1}, we conclude that $Du_{(a_j)}$ and $m_{(a_j)}$ are Lipschitz in $T^{\epsilon}_{I}(a_j)$.
   
    By differentiating \eqref{mequ2} with respect to $x$, we obtain
    \begin{align*}
        D_xy= I_k + \frac{1}{K}\diag{|D_iu_{(a_j)}|^{p_i-2}}D^2u_{(a_j)}.
    \end{align*}
    Using the fact that
     \begin{align*}
        -\frac{1}{K}\diag{|D_i\varphi_{(a_j)}|^{\frac{p_i-2}{2}}}D^2\varphi_{(a_j)} \diag{|D_i\varphi_{(a_j)}|^{\frac{p_i-2}{2}}} =I_k ,
    \end{align*}
    and by the same argument as in \textbf{Case 1}, we get
    \begin{align*}
        \det D_xy&= \det \left(I_k + \frac{1}{K}\diag{|D_iu_{(a_j)}|^{\frac{p_i-2}{2}}}D^2u_{(a_j)}\diag{|D_iu_{(a_j)}|^{\frac{p_i-2}{2}}} \right) \\
        &= \det \left( \frac{1}{K}\diag{|D_iu_{(a_j)}|^{\frac{p_i-2}{2}}}(D^2u_{(a_j)}-D^2\varphi_{(a_j)}) \diag{|D_iu_{(a_j)}|^{\frac{p_i-2}{2}}} \right)=: \det B,
    \end{align*}
    and $B \geq 0$ using \eqref{gradhessequ}.
    Therefore,
    \begin{align*}
        \lambda \tr{B} &= \mathcal{M}^-(B) \\
        & \leq \mathcal{M}^{+}(I_k) + \frac{1}{K}\mathcal{M}^{-}\left(\diag{|D_iu_{(a_j)}|^{\frac{p_i-2}{2}}}D^2u_{(a_j)} \diag{|D_iu_{(a_j)}|^{\frac{p_i-2}{2}}}\right) \\
        &\leq k\Lambda +\frac{1}{K}\mathcal{M}^{-}\left(\diag{|D_iu|^{\frac{p_i-2}{2}}}D^2u \diag{|D_iu|^{\frac{p_i-2}{2}}}\right) \\
        & \leq k\Lambda + \frac{1}{K}(\epsilon_0 + \Lambda\sum|D_iu|^{p_i-1}) \\
        &\leq C(1+\epsilon_0)
    \end{align*}
    where we used $D_j u(x)=0$ for any $j\in J$ and $x\in T_I^\epsilon$.
    Thus, we get
    \begin{align*}
        \det D_xy = \det B \leq \left( \frac{\tr{B}}{k}\right)^k \leq C(1+\epsilon_0)^k.
    \end{align*}
    Applying Fubini's theorem and the area formula to $m_{(a_j)}$, we obtain
    \begin{align*}
		|V^\epsilon_I|&= \int_{\prod_{j\in J}[-M_0^{\beta_j},M_0^{\beta_j}]}|V_I^\epsilon(a_j)| \  da_J  \\
		&\leq \int_{\prod_{j\in J}[-M_0^{\beta_j},M_0^{\beta_j}]}\int_{T^\epsilon_I(a_j)} |\det D_xy| \ dx_I da_J \\
		& \leq \int_{\prod_{j\in J}[-M_0^{\beta_j},M_0^{\beta_j}]}\int_{T^\epsilon_I(a_j)} C(1+\epsilon_0)^{k} \ dx_I da_J \\
		& = \int_{T^\epsilon_I} C(1+\epsilon_0)^{k} \ dx \\
		&\leq C\left(1+ \epsilon_0\right)^k|T^\epsilon_I|.
    \end{align*}
    Letting $\epsilon \rightarrow0$, we conclude that $|V_I| = \lim_{\epsilon\rightarrow0}|V_I^\epsilon| \leq \lim_{\epsilon\rightarrow0}C\left(1+ \epsilon_0\right)^k|T^\epsilon_I| \leq C\left(1+ \epsilon_0\right)^k|T_I|$.

    \textbf{ Case 3 : $I= I_2= \emptyset$.} If $x \in T_\emptyset$,  we have $y=x$ where $y \in V_\emptyset$ is the corresponding vertex.
	Thus we have $V_\emptyset = T_\emptyset$.

    \textbf{Conclusion.}
		Combining all the previous results, we have $$|V_I| \leq C\left(1+ \epsilon_0\right)^n|T_I|\quad \text{ for every }I_2 \subset I \subset \{1,\cdots,n\}.$$
		Using \eqref{TV}, we obtain
		\begin{align*}
			|K_{r_0}(1)|=|V| \leq \sum_{I_2 \subset I \subset \{1,\cdots,n\}} |V_I| 
            \leq \sum_{I_2 \subset I \subset \{1,\cdots,n\}} C\left(1+ \epsilon_0\right)^n|T_I| = C\left(1+ \epsilon_0\right)^n|T|.
		\end{align*}
    Since $T \subset \{u<M_0\} \cap K_1(M_0)$, it follows that
    \begin{align*}
        |\{u<M_0\} \cap K_1(M_0)| \geq |T| \geq \frac{|K_{r_0}(1)|}{C(1+\epsilon_0)^n} = \delta_0|K_1(M_0)|
    \end{align*}
    for $\delta_0 := \frac{|K_{r_0}(1)|}{C(1+\epsilon_0)^n|K_1(M_0)|} \in (0,1)$.
    This proves the lemma.
\end{proof}

\section{A barrier function} \label{sec4}
In this section, we construct the max-type anisotropic barrier under the exponent condition \eqref{pcond}.
We emphasize that this is the only part of the paper where \eqref{pcond} is used.
Consequently, if a suitable barrier could be constructed under a weaker condition on the exponents $(p_i)$, the rest of the argument would yield the Harnack inequality under the same weaker condition.

Our aim is to construct a positive barrier function $\Phi\in C^1(\mathbb R^n\setminus\{0\})$ satisfying
\begin{align*}
    \mathcal{M}^-_{\lambda,\Lambda} \left( \diag{|D_i\Phi|^{\frac{p_i-2}{2}}} D^2\Phi \diag{|D_i\Phi|^{\frac{p_i-2}{2}}} \right) -\mu\sum_i|D_i\Phi|^{p_i-1} \gtrsim \Phi^{q}
\end{align*}
for some $q>0$.
To motivate the construction, we introduce an auxiliary scalar variable $t$, which we regard heuristically as the size of the barrier.
To balance the different directional homogeneities of the operator, we consider an anisotropic function
\begin{align*}
    F(x,t) = \sum_i d_i |x_i|^{a_i} t^{b_i},
\end{align*}
where the exponents $a_i,b_i>0$ and the coefficients $d_i>0$ are to be chosen.

We choose the exponents $a_i$ and $b_i$ so that the directional terms $|D_iF|^{p_i-2}D_{ii}F$ have the same homogeneity for every $i$.
Indeed,
\begin{align*}
    |D_iF|^{p_i-2}D_{ii}F \simeq |x_i|^{a_i(p_i-1)-p_i} t ^{b_i(p_i-1)}.
\end{align*}
Thus, to eliminate the spatial factor $x_i$ and make all these quantities comparable to the same power $t^q$, independently of $i$, we choose
\begin{align} \label{aibi}
    a_i = \frac{p_i}{p_i-1} >0\quad \text{and} \quad b_i = \frac{q}{p_i-1} > 1
\end{align}
We require $b_i>1$ for every $i$, and hence choose $q >p_n-1$.
It remains to convert the auxiliary scale $t$ into an actual function of $x$.
To this end, we define the barrier function as
\begin{align} \label{Phidef}
    \Phi(x) :=  \max_{t\ge0}\left\{ t-F(x,t)\right\} = \max_{t\geq0}\left\{ t-\sum_i d_i |x_i|^{\frac{p_i}{p_i-1}}t^{\frac{q}{p_i-1}}\right\}.
\end{align}
Since $t-F(x,t)$ is strictly concave for $x\neq0$, the maximum is attained at a unique point $t=T(x)>0$.
Moreover, we have $ \Phi(x)= T(x)-F(x,T(x)) < T(x)$.
Thus, the scale $T^q$ arising from the construction above dominates the desired quantity $\Phi^q$.
Furthermore, the maximization also preserves the differential structure built into $F$.
Indeed, we will show that
    \begin{align*}
        D_i\Phi(x) = -D_iF(x,T(x)),
    \end{align*}
so that derivatives of $T(x)$ do not appear in the expression for $D\Phi$.
This allows the differential structure built into $F$ to pass naturally to the barrier $\Phi$.

Note that by the implicit function theorem, $T$ is $C^1(\R^n\setminus\{0\})$ and hence $\Phi \in C^1(\R^n\setminus\{0\})$.
However, $\Phi$ may fail to be $C^2$ on the coordinate hyperplanes $\Sigma =\bigcup_{\{i:p_i>2\}}\{x_i=0\}$ since $D_{ii}F$ may become singular there.

The main result of this section is the following lemma, which shows that $\Phi$ serves as a viscosity subsolution of the following equation.

\begin{lem} \label{Phipdelem}
    Assume that $(p_i)$ satisfies \eqref{pcond}.
    Then there exist $q>p_n-1$ and $d_i>0$, depending only on $n,(p_i),\lambda$, and $\Lambda$, such that the function $\Phi$ defined in \eqref{Phidef} has the following property:
    for any $R>1$, there exists $\mu\in(0,\Lambda)$ depending on $n,(p_i),\lambda,\Lambda$ and $R$, such that
    \begin{align} \label{Phiequ}
        \mathcal{M}^-_{\lambda,\Lambda} \left( \diag{|D_i\Phi|^{\frac{p_i-2}{2}}} D^2\Phi \diag{|D_i\Phi|^{\frac{p_i-2}{2}}} \right) -\mu\sum_i|D_i\Phi|^{p_i-1} > \Phi^{q} \quad \text{ in }Q_{R} \setminus\{0\}
    \end{align}
    in the viscosity sense.
\end{lem}
\begin{proof}
    By direct differentiation,
    \begin{align*}
        D_i F &= \sgn{x_i}d_i\frac{p_i}{p_i-1}|x_i|^{\frac{1}{p_i-1}} t^{\frac{q}{p_i-1}}, & D_{ii}F &= d_i\frac{p_i}{(p_i-1)^2}|x_i|^{-\frac{p_i-2}{p_i-1}} t^{\frac{q}{p_i-1}},  \\
        D_{it} F &= \sgn{x_i}d_i \frac{p_iq}{(p_i-1)^2} |x_i|^{\frac{1}{p_i-1}}t^{\frac{q}{p_i-1}-1}, & D_{tt}F &=\sum_{i} d_i\frac{q(q-p_i+1)}{(p_i-1)^2}|x_i|^{\frac{p_i}{p_i-1}} t^{\frac{q}{p_i-1}-2}.
    \end{align*}
    Since $q-p_i+1>0$ for every $i$, we have $D_{tt}F>0$ for any $x\neq0$, so that $t-F(x,t)$ is strictly concave in $t$.
    Moreover, since $\frac{q}{p_i-1}>1$, \begin{align*}
        \lim_{t\rightarrow0}\frac{F(x,t)}{t} =0, \quad \lim_{t\rightarrow\infty}\frac{F(x,t)}{t} =\infty,
    \end{align*}
    for any $x\neq0$.
    Hence, $t-F(x,t)>0$ for sufficiently small $t>0$, while $t-F(x,t) \rightarrow -\infty$ as $t\rightarrow \infty$.
    
    Therefore, for every $x\neq0$, there exists a unique maximizer $T(x)>0$ in \eqref{Phidef}. 
    In particular,
    \begin{align} \label{PhiT}
        0 < \Phi(x) = T(x)-F(x,T(x)) < T(x).
    \end{align}
    Note that since $t=T(x)$ is a maximum point, 
    \begin{align} \label{DtF}
        D_tF(x,T(x)) =1.
    \end{align}
    Furthermore, $D_tF$ is $C^1$ and $D_{tt}F(x,T(x))>0$ for $x\neq0$.
    The implicit function theorem therefore yields $T \in C^1(\R^n \setminus\{0\})$, and consequently $\Phi \in C^1(\R^n \setminus\{0\})$.
    Moreover, differentiating $\Phi(x)=T(x)-F(x,T(x))$ and using \eqref{DtF}, we obtain
    \begin{align} \label{DPhi}
        D_i\Phi(x) = -(D_iF(x,T(x)) + (D_tF(x,T(x))-1) D_iT ) = -D_iF(x,T(x)),
    \end{align}
    Differentiating \eqref{DtF} with respect to $x_i$ gives
    \begin{align} \label{Tdiff}
        D_iT = - \frac{D_{it}F}{D_{tt}F}.
    \end{align}
    Hence, differentiating \eqref{DPhi} with respect to $x_j$ and using \eqref{Tdiff}, we obtain
    \begin{align} \label{D2Phi}
        D_{ij} \Phi(x) &=   -D_{it}F(x,T(x)) D_j T-D_{ij} F(x,T(x)) \\
        &=\left(\frac{D_{it}F D_{jt}F}{D_{tt}F}- D_{ij} F \right)(x,T(x)).
    \end{align}
    If $p_i=2$, then by \eqref{DPhi},
    $D_i\Phi(x) = -2d_ix_iT(x)^q$, which is $C^1$.
    Hence no loss of second-order differentiability occurs in the $i$-th coordinate direction, even on $\{x_i=0\}$.
    However, if $p_j>2$, then $\frac{1}{p_j-1}<1$, which implies that $D_{jj}\Phi$ may become singular on $\{x_j=0\}$.
    Therefore, we carry out the following computation away from the singular coordinate hyperplane, $\Sigma:=\bigcup_{\{i:p_i>2\}}\{x_i=0\}.$

    For $x \in \R^n \setminus (\Sigma \cup \{0\})$,  \eqref{DPhi} and \eqref{D2Phi} yield
    \begin{align} \label{DPhiD2Phi}
        \diag{|D_i\Phi|^{\frac{p_i-2}{2}}} D^2\Phi \diag{|D_i\Phi|^{\frac{p_i-2}{2}}}  = v \otimes v-A, 
    \end{align}
    where
    \begin{align} \label{Avdef}
        A= \diag{|D_iF|^{p_i-2}D_{ii}F}, \quad v=\left(\frac{|D_iF|^{\frac{p_i-2}{2}}D_{it}F}{\sqrt{D_{tt}F}} \right),
    \end{align}
    with all quantities on the right-hand side evaluated at $(x,T)$.
    By the choice of exponents in \eqref{aibi},
    \begin{align*}
        |D_iF|^{p_i-2}D_{ii}F &= \frac{(d_ip_i)^{p_i-1}}{(p_i-1)^{p_i}} T^q,
    \end{align*}
    In particular, each entry of $A$ has the same $T^q$-homogeneity, independently of the coordinate direction $i$.
    Moreover,
    \begin{align*}
        |D_iF|^{p_i-2}|D_{it}F|^2 &= \frac{q^2(d_ip_i)^{p_i}}{(p_i-1)^{p_i+2}} |x_i|^{\frac{p_i}{p_i-1}} T^{q + \frac{q}{p_i-1}-2}.
    \end{align*}
    The powers of $x_i$ and $T$ in this expression agree with those in the corresponding term of $T^qD_{tt}F$.
    We therefore choose $d_i>0$ so that their coefficients differ by the same constant $K>0$, namely
    \begin{align*}
        \frac{q^2(d_ip_i)^{p_i}}{(p_i-1)^{p_i+2}} = K\frac{d_iq(q-p_i+1)}{(p_i-1)^2}.
    \end{align*}
    Equivalently,
    \begin{align*}
        d_i = \frac{(p_i-1)^{\frac{p_i}{p_i-1}}}{p_i} \left( K\frac{q+1-p_i}{qp_i} \right)^{\frac{1}{p_i-1}},
    \end{align*}
    With this choice of $d_i$, we obtain 
    \begin{align*}
        |D_iF|^{p_i-2}D_{ii}F &= K\frac{q-p_i+1}{qp_i} T^q \\
        \sum_i|D_iF|^{p_i-2}|D_{it}F|^2 &= KT^q D_{tt}F,
    \end{align*}
    and consequently
    \begin{align*}
        |v|^2 = \sum_{i} \frac{|D_iF|^{p_i-2}|D_{it}F|^2}{D_{tt}F} =KT^q.
    \end{align*}
    Therefore, since $A \geq 0$ and $ v\otimes v \geq0$, Proposition \ref{pucciprop} gives
    \begin{align}
        \mathcal{M}^- \left( \diag{|D_i\Phi|^{\frac{p_i-2}{2}}} D^2\Phi \diag{|D_i\Phi|^{\frac{p_i-2}{2}}} \right) \nonumber
        &\geq \mathcal{M}^-  \left( v \otimes v \right) - \mathcal{M}^+\left( A \right) \nonumber\\
        & \geq \lambda|v|^2 - \Lambda\sum_i A_{ii} \nonumber\\
        & = KT^q \left(\lambda -\Lambda\sum \frac{q-p_i+1}{qp_i}\right) \nonumber\\
        & = KT^q \Lambda q^{-1} \left( \left(\frac{\lambda}{\Lambda} - \frac{n}{\bar{p}}\right)q +n-\frac{n}{\bar{p}}\right) =: K T^q\kappa.  \label{MDPhi}
    \end{align}
    To ensure that $\kappa>0$, we distinguish two cases.
    If
    \begin{align*}
        \frac{\lambda}{\Lambda} - \frac{n}{\bar{p}} \geq0,
    \end{align*}
    which is equivalent to the second condition in \eqref{pcond}, then $\kappa>0$ for every $q>p_n-1$.
    Otherwise, $\kappa>0$ requires
    \begin{align*}
         q< \frac{n(\bar{p}-1)}{n-\frac{\lambda}{\Lambda}\bar{p}}.
    \end{align*}
    Together with \eqref{aibi}, we therefore need to choose
    \begin{align*}
        p_n-1<q< \frac{n(\bar{p}-1)}{n-\frac{\lambda}{\Lambda}\bar{p}}.
    \end{align*}
    Such a choice is possible if and only if
    \begin{align*}
        p_n-1< \frac{n(\bar{p}-1)}{n-\frac{\lambda}{\Lambda}\bar{p}},
    \end{align*}
    which is equivalent to the first condition in \eqref{pcond}.

    It remains to control the first-order term.
    For $x \in Q_R$, we have
    \begin{align*}
       \sum_i |D_i\Phi|^{p_i-1} = KT^q\sum_i\frac{(p_i-1)(q-p_i+1)}{qp_i}|x_i| \leq KRT^q\sum_i\frac{(p_i-1)(q-p_i+1)}{qp_i}.
    \end{align*}
    We choose
    \begin{align*}
        K > \frac{2}{\kappa}, \quad \text{and} \quad \mu \leq  \frac{\kappa}{2R\sum_i\frac{(p_i-1)(q-p_i+1)}{qp_i}}.
    \end{align*}
    Then, using \eqref{MDPhi} and the estimates above, we obtain
    \begin{align*}
        \mathcal{M}^- \left( \diag{|D_i\Phi|^{\frac{p_i-2}{2}}} D^2\Phi \diag{|D_i\Phi|^{\frac{p_i-2}{2}}} \right)  - \mu\sum_i |D_i\Phi|^{p_i-1} &\geq KT^q\left(\kappa-\mu R\sum_i\frac{(p_i-1)(q-p_i+1)}{qp_i}\right) \\
&> T^q > \Phi^q,
    \end{align*}
    where we used \eqref{PhiT} in the last inequality.
    Thus, \eqref{Phiequ} holds classically in $Q_R\setminus(\Sigma\cup\{0\})$.
    
    It remains to verify the inequality on $\Sigma$ in the viscosity sense.
    Suppose that $\phi\in C^2$ touches $\Phi$ from above at a point $y \in \Sigma$, and set $J=\{j:y_j=0,p_j>2\}$.
    Then $D_j \phi(y) =D_j \Phi(y)=0$ for $j\in J$, so that the $j$-th derivative terms such as $D_{jj}\phi$ do not contribute to the operator.
    Moreover, on the slice $S=\{z:z_j=0 \text{ for every } j\in J\}$, $\Phi$ is $C^2$ and $D^2\phi|_S(y) \geq D^2\Phi|_S(y)$.
    Therefore, by the ellipticity of $\mathcal M^-{\lambda,\Lambda}$, we may apply the preceding computation on $S$ with the directions in $J$ omitted.
    Then the identity $|v|^2=KT^q$ remains unchanged, whereas the negative contribution from $A$ becomes $\sum_{i\notin J}A_{ii}\le\sum_i A_{ii}$.
    Therefore, the lower bound obtained above remains valid, and $\phi$ satisfies the required inequality at $y$.
    Consequently, $\Phi$ is a viscosity subsolution of \eqref{Phiequ} in $Q_{R} \setminus\{0\}$.

\end{proof}

We also require the following scaling property of the barrier function $\Phi$.
\begin{lem} \label{barscal}
    For any $r>0$, we have
    \begin{align*}
        \lim_{L\rightarrow\infty} \frac{1}{L} \sup_{\R^n \setminus K_r(L)} \Phi =0, \quad \lim_{m\rightarrow0} \frac{1}{m} \inf_{ K_r(m)} \Phi =\infty.
    \end{align*}
\end{lem}
\begin{proof}
    For $s>0$, the function $F$ satisfies
    \begin{align*}
        F\left( \left(s^{\frac{-(q-p_i+1)}{p_i}}x_i\right),st\right) = sF(x,t).
    \end{align*}
    Therefore, we have
    \begin{align*}
        \Phi \left(s^{\frac{-(q-p_i+1)}{p_i}}x_i\right) &= \max_{st>0} \left\{ st- F\left( \left(s^{\frac{-(q-p_i+1)}{p_i}}x_i\right),st\right)\right\} \\
        &=\max_{t>0} \left\{ st- sF(x,t) \right\} = s\Phi(x).
    \end{align*}
    For $x\neq0$, set $s = \max_{i} |x_i|^{\frac{p_i}{q-p_i+1}}$.
    Then we have $y=\left(s^{\frac{-(q-p_i+1)}{p_i}}x_i\right) \in \partial Q_1$ and $\Phi(x) = \frac{\Phi(y)}{s}$. Therefore, we obtain
    \begin{align} \label{Phicomp}
        \frac{\min_{\partial Q_1}\Phi}{\max_i|x_i|^{\frac{p_i}{q-p_i+1}} } \leq \Phi(x) \leq \frac{\max_{\partial Q_1}\Phi}{\max_i|x_i|^{\frac{p_i}{q-p_i+1}} },
    \end{align}
    for any $x \in \R^n \setminus\{0\}$.
    Recalling the definition of $K_r(M)$ in \eqref{intcubedef}, \eqref{Phicomp} yields
    \begin{align*}
        \frac{1}{L} \sup_{\R^n \setminus K_r(L)} \Phi & \leq \frac{C}{L}  \max_i\left\{ r^{-\frac{1}{q-p_i+1} } L^{\frac{p_n-p_i}{q-p_i+1}}\right\} \\
        & \leq C  \max_i\left\{ r^{-\frac{1}{q-p_i+1} } L^{-\frac{q-p_n+1}{q-p_i+1}}\right\} 
        \overset{L\rightarrow \infty}{\longrightarrow} 0,
    \end{align*}
    since $q-p_i+1 \geq q-p_n+1  >0$. Similarly,
    \begin{align*}
        \frac{1}{m} \inf_{ K_r(m)} \Phi & \geq  \frac{c}{m} \min_i\left\{ r^{-\frac{1}{q-p_i+1} } m^{\frac{p_n-p_i}{q-p_i+1}}\right\}\\
        &\geq c  \min_i\left\{ r^{-\frac{1}{q-p_i+1} } m^{-\frac{q-p_n+1}{q-p_i+1}}\right\}
        \overset{m\rightarrow 0}{\longrightarrow} \infty.
    \end{align*}
\end{proof}

\section{Doubling property} \label{sec5}
In this section, we establish the doubling property for viscosity supersolutions using the barrier function constructed in Section \ref{sec4}.
Throughout Sections \ref{sec5}-\ref{sec7}, we assume that the exponents $(p_i)$ satisfy the assumption \eqref{pcond}.
\begin{lem} \label{doublem}
    There exist $m_0 \in(0,1)$, $L_0>1$, $R_1>1$, $\mu_0 \in(0,1)$ and $\epsilon_0 \in (0,1)$ such that 
    if $u \in C\left(\overline{K_{R_1}(m_0)}\right)$ is nonnegative in $K_{R_1}(m_0)$ and satisfies
    \begin{align} \label{doubpde}
            \mathcal{M}^- \left( \diag{|D_iu|^{\frac{p_i-2}{2}}} D^2u \diag{|D_iu|^{\frac{p_i-2}{2}}} \right) -\mu_0\sum|D_iu|^{p_i-1} \leq  \epsilon_0 \quad \text{ in } K_{R_1}(m_0), 
    \end{align}
    then
    \begin{align*}
        u>L_0m_0 \quad \text{ in }K_{r_0}(L_0m_0), \quad \Longrightarrow \quad u>m_0 \quad \text{ in } K_1(m_0),
    \end{align*}
    where $r_0  \in(0,1)$ is as in Lemma \ref{basicmeaslem}.
\end{lem}
\begin{proof}
    We construct a barrier function $\Psi \in C(K_{R_1}(m_0) \setminus K_{r_0}(L_0m_0))$ that satisfies the following properties:
    \begin{enumerate}
        \item $\mathcal{M}^- \left( \diag{|D_i\Psi|^{\frac{p_i-2}{2}}} D^2\Psi \diag{|D_i\Psi|^{\frac{p_i-2}{2}}} \right) -\mu_0\sum|D_i\Psi|^{p_i-1} -\epsilon_0\Psi>  \epsilon_0$ in $K_{R_1}(m_0) \setminus K_{r_0}(L_0m_0)$.
        \item $\Psi < 0  \leq u$ on $\partial K_{R_1}(m_0)$.
        \item $\Psi < L_0m_0 \leq u$ on $\partial K_{r_0}(L_0m_0)$.
        \item $\Psi > m_0$ in $ K_{1}(m_0)$.
    \end{enumerate}
    By Lemma \ref{barscal}, there exists $m_0 \in(0,1)$ such that $$\frac{1}{m_0} \inf_{ K_1(m_0)} \Phi >2.$$
    We define the barrier as a shift of $\Phi$: 
    $$\Psi(x) = \Phi(x) - m_0.$$
    Since $\Phi >2m_0$ in $K_1(m_0)$, condition $(4)$ follows.
    Moreover, we choose $R_1>1$ sufficiently large such that $\{ \Phi \geq m_0\} \subset K_{R_1}(m_0)$, which implies condition $(2)$.
    We also choose small $\eta \in(0,1)$ such that $ K_{R_1}(m_0) \subset \{\Phi \geq \eta m_0\}$.
    Using Lemma \ref{barscal} again, we choose $L_0>\frac{1}{m_0}$ sufficiently large satisfying $$\frac{1}{L_0m_0} \sup_{\R^n \setminus K_{r_0}(L_0m_0)} \Phi <1,$$ which guarantees condition $(3)$.
    Finally, by Lemma \ref{Phipdelem}, there exists a constant $\mu_0>0$ such that
    \begin{align*}
        \mathcal{M}^-\left( \diag{|D_i\Psi|^{\frac{p_i-2}{2}}} D^2\Psi \diag{|D_i\Psi|^{\frac{p_i-2}{2}}} \right) -\mu_0\sum_i|D_i\Psi|^{p_i-1} &> \Phi^{q} \geq  (\eta m_0)^q \geq\epsilon_0(1+L_0m_0) \\
        & \geq \epsilon_0 (1+\Psi)\quad \text{ in }K_{R_1}(m_0) \setminus \overline{K_{r_0}(L_0m_0)},
    \end{align*}
    where $\epsilon_0 = \frac{\eta^{q} m_0^{q}}{2(1+L_0m_0)}$.
    This yields condition $(1)$.
    Note that since $u \geq0$, we have
    \begin{align*}
        F(u,Du,D^2u) :=  \mathcal{M}^- \left( \diag{|D_iu|^{\frac{p_i-2}{2}}} D^2u \diag{|D_iu|^{\frac{p_i-2}{2}}} \right) -\mu_0\sum|D_iu|^{p_i-1} -\epsilon_0u\leq  \epsilon_0 \quad \text{ in }K_{R_1}(m_0).
    \end{align*}
    Since $F(r,p,X)$ is strictly decreasing in $r$ and degenerate elliptic, $F$ satisfies the comparison principle (Lemma \ref{comp}).
    Thus, since $F(u,Du,D^2u)\leq\epsilon_0 \leq  F(\Psi,D\Psi,D^2\Psi) $ in $K_{R_1}(m_0) \setminus \overline{K_{r_0}(L_0m_0)}$ and $\Psi \leq u$ on  $\partial(K_{R_1}(m_0) \setminus K_{r_0}(L_0m_0))$, we have $\Psi \leq u$ in  $K_{R_1}(m_0) \setminus K_{r_0}(L_0m_0)$.
    Therefore, we obtain $ m_0 < \Psi \leq u$ in $K_1(m_0)\setminus K_{r_0}(L_0m_0)$, which completes the proof.
\end{proof}

By iterating the previous result, we obtain the following result.
\begin{lem} \label{doub2lem}
    Let $m_0 \in(0,1)$, $L_0>1$, $R_1>1$, $\mu_0\in (0,1)$ and $\epsilon_0 \in(0,1)$ be as in the previous lemma.
    If $u \in C(K_{R_1}(m_0))$ is nonnegative and satisfies \eqref{doubpde} in $K_{R_1}(m_0)$, then
    \begin{align} \label{doubkres}
        u>L_0^km_0 \quad \text{ in }K_{r_0^k}(L_0^km_0), \quad \Longrightarrow \quad u>m_0 \quad \text{ in } K_1(m_0)
    \end{align}
    for any $k \in \N$.
\end{lem}
\begin{proof}
    We proceed by induction on $k$.
    The case of $k=1$ follows directly from Lemma \ref{doublem}.
    We assume that \eqref{doubkres} is true for some $k \in\N$ and prove it for $k+1$.
    Suppose that $u>L_0^{k+1}m_0$ in $K_{r_0^{k+1}}(L_0^{k+1}m_0)$.
    We define a scaled function
    \begin{align*}
        v(x) = \frac{u((r_0^k)^{\alpha_i}(L_0^k)^{\beta_i} x_i)}{L_0^k}.
    \end{align*}
    The scaled function $v$ is defined in $ K_{R_1/r_0^k}(m_0/L_0^k)$.
    Observe that $K_{R_1}(m_0)\subset K_{R_1/r_0^k}(m_0/L_0^k)$ and, by Remark \ref{scal}, the function $v$ is again nonnegative and satisfies \eqref{doubpde} in $K_{R_1}(m_0)$.
    Moreover, the inductive assumption $u>L_0^{k+1}m_0$ in $K_{r_0^{k+1}}(L_0^{k+1}m_0)$ translates directly to $v > L_0m_0$ in $K_{r_0}(L_0m_0)$.
    Applying Lemma \ref{doublem}, we obtain $v > m_0$ in $K_{1}(m_0)$.
    Rescaling back, this implies $u>L_0^km_0$ in $K_{r_0^k}(L_0^km_0)$.
    Therefore, by the induction assumption that \eqref{doubkres} is true for $k$, we conclude that $u>m_0$ in $K_1(m_0)$, which completes the proof.
\end{proof}

We conclude this section by proving the following type of doubling property.
\begin{lem} \label{doub3lem}
    Let $m_0 \in(0,1)$, $R_1>1$, $\mu_0\in (0,1)$ and $\epsilon_0 \in(0,1)$ be as in the previous lemma.
    For any fixed $\nu >1$, there exist $r_1 \in (0,1)$ and $k_0 \in \N$ such that if $u \in C(K_{R_1}(m_0))$ is nonnegative  and satisfies \eqref{doubpde} in $K_{R_1}(m_0)$, then
    \begin{align} \label{doubkres2}
        u>\nu^km_0 \quad \text{ in }K_{r_1^k}(\nu^km_0), \quad \Longrightarrow \quad u>m_0 \quad \text{ in } K_1(m_0),
    \end{align}
    for any $k \geq k_0$. 
\end{lem}
\begin{proof}
    Choose $q \in \N$ such that $\nu < L_0^q$.
    Set $L=L_0^q$ and $r=r_0^q$.
    By Lemma \ref{doub2lem}, we have
    \begin{align} \label{doubkres3}
        u>L^jm_0 \quad \text{ in }K_{r^j}(L^jm_0), \quad \Longrightarrow \quad u>m_0 \quad \text{ in } K_1(m_0),
    \end{align}
    for any $j \in \N$.
    Choose $j_0 \in \N$ sufficiently large so that $r^{j_0} L^{p_n-p_1} < 1$.
    Next, we define $r_1 = \nu^{-\theta} \in (0,1) $ with a sufficiently small $\theta \in (0,1)$ such that
    \begin{align} \label{thetacond}
        rL^{\theta} <1, \quad \text{ and } \quad r^{j_0}L^{\theta(j_0+1)} L^{p_n-p_1}\leq 1.
    \end{align}
    We select $k_0 \in \N$ such that $L^{j_0} \leq \nu^{k_0}  < L^{j_0+1}$.
    Then for any $k \geq k_0$, there exists $j \geq j_0$ such that $L^{j} \leq \nu^{k}  < L^{j+1}$.
    Therefore, since $K_r(N) \supset K_r(M)$ for any $0<N\leq M$ and $r>0$, we have
    \begin{align*}
        u>\nu^km_0 \geq L^{j}m_0 \quad \text{ in }K_{r_1^k}(\nu^km_0) \supset K_{r_1^k}(L^{j+1}m_0).
    \end{align*}
    We claim that $K_{r_1^k}(L^{j+1}m_0) \supset K_{r^{j}}(L^{j}m_0)$.
    Observe that by \eqref{thetacond} and the fact that $j \geq j_0$, we obtain
    \begin{align*}
        \frac{r^{j}}{r_1^k}L^{p_n-p_i}=r^{j}\nu^{\theta k}L^{p_n-p_i} \leq r^{j}L^{\theta(j+1)} L^{p_n-p_i}\leq 1,
    \end{align*}
    for any coordinate index $i$.
    Since $-\frac{\beta_i}{\alpha_i}=p_n-p_i$, this is equivalent to
    \begin{align*}
        (r_1^k)^{\alpha_i} (L^{j+1}m_0)^{\beta_i} \geq (r^{j})^{\alpha_i} (L^{j}m_0)^{\beta_i},
    \end{align*}
    which implies the claim.
    Therefore, we have $u> L^{j}m_0$ in $K_{r^{j}}(L^{j}m_0)$.
    Applying \eqref{doubkres3}, we conclude that $u >m_0$ in $K_1(m_0)$.
\end{proof}

\section{$L^\epsilon$ estimate} \label{sec6}
In this section, we establish an $L^\epsilon$ estimate.
We first prove the following lemma by combining the basic measure estimate (Lemma \ref{basicmeaslem}) and the doubling property (Lemma \ref{doublem}).

\begin{lem} \label{measbarlem}
    Let $m_0 \in(0,1)$, $R_1>1$, $\mu_0\in (0,1)$ and $\epsilon_0 \in(0,1)$ be as in Lemma \ref{doublem}.
    There exist $L_1\geq L_0$ and $\delta_1 \in(0,1)$ such that if $r \in(0,1)$ and $u \in C(K_{R_1}(m_0))$ is nonnegative and  satisfies \eqref{doubpde} in $K_{R_1}(m_0)$ with
    \begin{align*}
        \inf_{K_r(m_0)} u \leq m_0,
    \end{align*}
    then
    \begin{align*}
        |\{u <L_1m_0\} \cap K_r(m_0)\}| \geq \delta_1|K_r(m_0)|.
    \end{align*}
\end{lem}
\begin{proof}
    We define a rescaled function $v \in C(K_{R_1/r}(m_0))$ as
    \begin{align*}
        v(x) = u(r^{\alpha_i} x_i).
    \end{align*}
    Then $v$ is nonnegative and satisfies \eqref{doubpde} in $K_{R_1}(m_0) \subset K_{R_1/r}(m_0)$ by Remark \ref{scal}.
    Moreover, since $\inf_{K_1(m_0)} v \leq m_0$, we have $\inf_{K_{r_0}(L_0m_0)} v \leq L_0m_0$ by Lemma \ref{doublem}.
    We also define a second scaled function $w \in C(K_{R_1}(1/L_0))$ as 
    \begin{align*}
        w(x) = \frac{v( (L_0m_0)^{\beta_i}x)}{L_0m_0}.
    \end{align*}
    Then $w$ is also nonnegative and satisfies \eqref{doubpde} in $K_1(M_0) \subset K_{R_1}(1/L_0)$.
    Furthermore, since $\inf_{K_{r_0}(L_0m_0)} v \leq L_0m_0$, we have $\inf_{K_{r_0}(1)} w \leq 1$.
    Applying Lemma \ref{basicmeaslem} yields
    \begin{align*}
        |\{w <M_0\} \cap K_1(M_0)\}| \geq \delta_0|K_1(M_0)|,
    \end{align*}
    for some $\delta_0=\delta_0(\epsilon_0) \in(0,1)$.
    Scaling back to the original function $u$, we obtain
    \begin{align*}
        |\{u <M_0L_0m_0\} \cap K_r(M_0L_0m_0)\}| \geq \delta_0|K_r(M_0L_0m_0)|.
    \end{align*}
    Letting $L_1 = M_0L_0$ and using $K_r(L_1m_0) \subset K_r(m_0)$, we have
    \begin{align*}
        |\{u <L_1m_0\} \cap K_r(m_0)\}| &\geq |\{u <L_1m_0\} \cap K_r(L_1m_0)\}| \\
        &\geq \frac{\delta_0|K_r(L_1m_0)|}{|K_r(m_0)|}|K_r(m_0)| = \delta_1|K_r(m_0)|,
    \end{align*}
    with $\delta_1 =\delta_0 \prod_{i}L_1^{\beta_i}$.
\end{proof}

Combining the previous lemma with the Vitali covering lemma (Lemma \ref{Vitalilem}), we now establish the $L^\epsilon$ estimate.
This result quantifies the decay of the superlevel sets for a supersolution $u$.
\begin{thm}[$L^\epsilon$ estimate] \label{Lepsilonthm}
    Let $m_0 \in(0,1)$, $L_1>1$, $\mu_0\in (0,1)$ and $\epsilon_0 \in(0,1)$ be as in the previous lemma.
    There exist $\epsilon \in(0,1)$, $R_2>1$ and $k_1 \in \N$ such that if $u \in C(K_{R_2}(m_0))$ is nonnegative in $K_{R_2}(m_0)$, and satisfies
    \begin{align} \label{Lepspde}
        \mathcal{M}^- \left( \diag{|D_iu|^{\frac{p_i-2}{2}}} D^2u \diag{|D_iu|^{\frac{p_i-2}{2}}} \right) -\mu_0\sum|D_iu|^{p_i-1} \leq  \epsilon_0 \quad \text{ in } K_{R_2}(m_0), 
    \end{align}
    with $u(0) \leq m_0$, then
    \begin{align*}
        \left|\{u  > L_1^km_0\} \cap \frac{2}{3}K_1(m_0)\right| \leq \epsilon^k\left|\frac{2}{3}K_1(m_0)\right| \quad \text{ for any } k\geq k_1.
    \end{align*}
\end{thm}
\begin{proof}
    We choose $R_2>R_1$ such that $R_2^{1/p_n} \geq R_1^{1/p_n}+1$.
    This choice ensures that all intrinsic cubes used below remain within the domain $K_{R_2}(m_0)$.
    We then fix $\rho >1$ satisfying $\rho(1-L_1^{\beta_m}) \geq 1$, where $\beta_m<0$ is the largest exponent among the nonzero exponents $\beta_i<0$.
    We also select $k_1 >k_0$ sufficiently large such that $\rho L_1^{k_1\beta_m} \leq \frac{1}{3}$ and $\sum_{l=k_1}^\infty r_1^{l/p_n} \leq \frac{1}{3}$, where $r_1=r_1(L_1) \in (0,1)$ and $k_0=k_0(L_1) \in \N$ are as in Lemma \ref{doub3lem}.
    We define shrinking cubes $Q_k$ as 
    \begin{align*}
        Q_k := \left\{|x_i| < m_0^{\beta_i}\left(\frac{2}{3}+\rho L_1^{k\beta_i} \right) \text{ for }\beta_i\neq 0, \quad  |x_j| < \frac{2}{3} + \sum_{l=k}^{\infty} r_1^{l/p_n}\text{ for }\beta_j = 0 \right\}.
    \end{align*}
    Then we have $Q_{k+1} \subset Q_k$, $Q_{k_1} \subset K_1(m_0)$ and $\lim_{k\rightarrow \infty} Q_k = \frac{2}{3}K_1(m_0)$.
    For $k \geq k_1$, we define a superlevel set
    \begin{align*}
        A_k := \{ u > L_1^k m_0\} \cap Q_k.
    \end{align*}
    For any $x_0 \in Q_{k+1} \cap A_k$, let $r_{x_0}>0$ be the largest radius such that $K_{r_{x_0}}(L_1^km_0,x_0) \subset A_k$. 
    We claim that
    \begin{align*}
        K_{r_1^k}(L_1^km_0,x_0) \subset Q_k, \quad \text{ and } \quad\inf_{K_{r_1^k}(L_1^km_0,x_0)}u \leq  L_1^km_0.
    \end{align*}
    This immediately implies $K_{r_1^k}(L_1^km_0,x_0) \not\subset A_k$, and therefore $r_{x_0} < r_1^k <1$.
    
    To prove the claim, first note that for $\beta_i \neq 0$,
    \begin{align*}
        m_0^{\beta_i}\left(\frac{2}{3} + \rho L_1^{(k+1)\beta_i} \right) + (r_1^k)^{\alpha_i} (L_1^km_0)^{\beta_i} \leq m_0^{\beta_i}\left(\frac{2}{3} +  L_1^{k\beta_i} (\rho L_1^{\beta_i} +1)\right) \leq m_0^{\beta_i}\left(\frac{2}{3} +  \rho L_1^{k\beta_i}\right),
    \end{align*}
    and for $\beta_j = 0$,
    \begin{align*}
        \left(\frac{2}{3} +  \sum_{l=k+1}^{\infty} r_1^{l/p_n} \right) + (r_1^{k})^{\alpha_j} \leq \left(\frac{2}{3} +  \sum_{l=k}^{\infty} r_1^{l/p_n} \right).
    \end{align*}
    Hence we get $ K_{r_1^k}(L_1^km_0,x_0) \subset Q_k$.
    For the second part, since $u(0) \leq m_0$ and $x_0 \in K_1(m_0)$, we have $\inf_{K_1(m_0,x_0)}u \leq m_0$.
    By the choice of $R_2$, we get $K_{R_1}(m_0,x_0) \subset K_{R_2}(m_0)$.
    Applying Lemma \ref{doub3lem} yields $\inf_{K_{r_1^k}(L_1^km_0,x_0)}u \leq L_1^km_0$, which completes the claim.
    
    Consequently, we obtain the collection of cubes $\mathcal{F} = \{K_{r_{x_0}}(L_1^km_0,x_0) \}_{x_0 \in Q_{k+1}\cap A_k}$ which is a covering of $Q_{k+1}\cap A_k$.
    Using the Vitali covering lemma (Lemma \ref{Vitalilem}), we find a countable disjoint subcollection $\mathcal{G} = \{K_l=K_{r_{l}}(L_1^km_0,x_l) \}_{l\in\N}$ such that $ \bigcup_l5K_l \supset Q_{k+1}\cap A_k$.
    By the construction, we have
    \begin{align*}
        K_l \subset A_k, \quad\text{ and }\quad \inf_{K_l} u = L_1^km_0.
    \end{align*}
    We now claim that for any $l \in\N$,
    \begin{align*}
        |K_l \cap A_k\setminus \tilde{A}_{k+1}| \geq \delta_1|K_l|,
    \end{align*}
    where
    \begin{align*}
        \tilde{A}_{k+1} = \{ u> L_1^{k+1}m_0 \} \cap Q_k.
    \end{align*}
    To prove this, we consider the rescaled function $v$ as
    \begin{align*}
        v(x) = \frac{u(x_l+(L_1^k)^{\beta_i}x)}{L_1^k}.
    \end{align*}
    Since $K_{R_1}(m_0,x_l) \subset K_{R_2}(m_0) $, $v$ is nonnegative and satisfies \eqref{doubpde} in $K_{R_1}(m_0)\subset K_{R_1}(m_0/L_1^k)$.
    Moreover, since $\inf_{K_{r_{l}}(L_1^km_0,x_l)} u \leq L_1^km_0$,  we have $\inf_{K_{r_l}(m_0)}v \leq m_0$.
    Applying Lemma \ref{measbarlem}, we have
    \begin{align*}
        |\{v <L_1m_0\} \cap K_{r_l}(m_0)\}| \geq \delta_1|K_{r_l}(m_0)|.
    \end{align*}
    Scaling back to $u$, this becomes
    \begin{align*}
        |\{u <L_1^{k+1}m_0\} \cap K_l\}| \geq \delta_1|K_l|,
    \end{align*}
    which proves the claim.
    
    Since the cubes $K_l$ are pairwise disjoint and $ \bigcup_l5K_l \supset Q_{k+1}\cap A_k$, we obtain
    \begin{align*}
        |Q_{k+1} \cap A_k| &\leq \sum_l |5K_l| \leq C \sum_l |K_l| \\
        &\leq C \sum_l |K_l \cap A_k \setminus\tilde{A}_{k+1}|  \leq  C|A_k\setminus\tilde{A}_{k+1}|\\
        &\leq C(|Q_{k+1} \cap A_k \setminus A_{k+1}| + |Q_k| - |Q_{k+1}|).
    \end{align*}
    Next, using the elementary inequality
    $$\prod_{i}(s_i+t_i) - \prod_{i} t_i \leq \sum_{i} s_i \prod_{j\neq i} (s_j+t_j),$$
    for any $s_i,t_i \geq 0$, together with $\frac{2}{3}K_1(m_0) \subset Q_k \subset K_1(m_0)$, we obtain
    \begin{align*}
        |Q_k|-|Q_{k+1}| &\leq C \left(\sum_{\{i: \beta_i\neq0\}} m_0^{\beta_i}(\rho L_1^{k\beta_i} - \rho L_1^{(k+1)\beta_i}) +\sum_{\{i: \beta_i=0\}} (\sum_{l=k}^{\infty} r_1^{l/p_n}-\sum_{l=k+1}^{\infty} r_1^{l/p_n})\right) \\
        &\leq C(L_1^{k\beta_m} + r_1^{k/p_n}).
    \end{align*}
    Combining the above estimates yields
    \begin{align*}
        |Q_{k+1} \cap A_{k+1}|  &= |Q_{k+1} \cap A_k| - |Q_{k+1} \cap A_k \setminus A_{k+1}| \\
        &\leq \left(1-\frac{1}{C}\right)|Q_{k+1} \cap A_k| + |Q_k| - |Q_{k+1}| \\
        &\leq \left(1-\frac{1}{C}\right)|Q_{k} \cap A_k| + C(L_1^{k\beta_m} + r_1^{k/p_n}).
    \end{align*}
    Applying a standard iteration argument to this recursive inequality completes the proof of the theorem.
\end{proof}

The following corollary is a direct consequence of Theorem \ref{Lepsilonthm} and will be used later in the proof of the Harnack inequality.
\begin{cor} \label{Lepscor}
    Under the same assumptions as in Theorem \ref{Lepsilonthm}, there exists $L_2>1$ such that
    \begin{align*}
        \left|\{u \geq L_2m_0\} \cap \frac{2}{3}K_1(m_0)\right| \leq \frac{1}{4^{n+1}}\left|\frac{2}{3}K_1(m_0)\right|.
    \end{align*}
\end{cor}

\section{Harnack inequality} \label{sec7}
In this section, we conclude the proof of the main result, Theorem \ref{Mainthm}, following an adaptation of the argument in \cite{Caffarelli95}.
\begin{lem} \label{Harnmidlem}
    Let $u \in C(K_{R_2}(m_0))$ be nonnegative in $K_{R_2}(m_0)$ and satisfy
    \begin{align} \label{harnpde}
    \begin{cases}
        \mathcal{M}^- \left( \diag{|D_iu|^{\frac{p_i-2}{2}}} D^2u \diag{|D_iu|^{\frac{p_i-2}{2}}} \right) -\mu_0\sum|D_iu|^{p_i-1} \leq  \epsilon_0 \\
        \mathcal{M}^+ \left( \diag{|D_iu|^{\frac{p_i-2}{2}}} D^2u \diag{|D_iu|^{\frac{p_i-2}{2}}} \right) +\mu_0\sum|D_iu|^{p_i-1} \geq  -\epsilon_0 
    \end{cases} 
     \quad \text{ in } K_{R_2}(m_0),
    \end{align}
    with $u(0) \leq m_0$.
    Then there exists $k_2 \in \N$ such that if
    there exists $x_0 \in K_1(m_0)$ with 
    \begin{align*}
        u(x_0) \geq 2\nu^{k-1} L_2m_0,
    \end{align*}
    for some $k \geq k_2$, then
    \begin{align*}
        \sup_{K_{R_{2}r_2^k}(\nu^km_0,x_0)} u \geq 2\nu^kL_2m_0,
    \end{align*}
    where $\nu = \frac{L_2}{L_2-1/2}$ and $r_2 =r_1(\nu)$ is as in Lemma \ref{doub3lem}.
\end{lem}

\begin{proof}
    Let $k_2 \geq \max \{k_1,k_0(\nu)\}$ with $1+r_2^{k_2\alpha_n}(R_2^{\alpha_n} +1) \leq R_2^{\alpha_n}$ ensuring all subsequent intrinsic cubes remain within the domain $K_{R_2}(m_0)$.
    We argue by contradiction and assume that for some $k \geq k_2$,
    \begin{align*}
        \sup_{K_{R_{2}r_2^k}(\nu^km_0,x_0)} u < 2\nu^kL_2m_0.
    \end{align*}
    Observe that the choice of $k_2$ implies $K_{R_{2}r_2^k}(\nu^km_0,x_0) \subset K_{R_2}(m_0)$.
    To derive a contradiction, we define a scaled function $v$ as follows:
    \begin{align*}
        v(x) = \frac{2\nu^kL_2m_0 - u(x_0+(r_2^k)^{\alpha_i}(\nu^k)^{\beta_i}x)}{\nu^k}.
    \end{align*}
    The function $v$ is nonnegative and satisfies \eqref{Lepspde} in $K_{R_2}(m_0)$.
    Since $v(0) \leq m_0$, we apply Corollary \ref{Lepscor} to find that
    \begin{align*}
        \left|\{v \geq L_2m_0\} \cap \frac{2}{3}K_1(m_0)\right| \leq \frac{1}{4^{n+1}}\left|\frac{2}{3}K_1(m_0)\right|.
    \end{align*}
    Scaling back to $u$ implies
    \begin{align} \label{harnmeas1}
        \left|\{u \leq \nu^k L_2m_0\} \cap \frac{2}{3}K_{r^k_2}(\nu^km_0,x_0)\right| \leq \frac{1}{4^{n+1}}\left|\frac{2}{3}K_{r^k_2}(\nu^km_0)\right|.
    \end{align}
    On the other hand, since $u(0) \leq m_0$, we have $\inf_{K_1(m_0,x_0)}u \leq m_0$.
    Using $K_{R_1}(m_0,x_0) \subset K_{R_2}(m_0)$ and applying Lemma \ref{doub3lem}, we obtain
    \begin{align*}
        \inf_{K_{r_2^k}(\nu^km_0,x_0)}u \leq \nu^k m_0.
    \end{align*}
    Consequently, there exists a point $\overline{x} \in K_{r_2^k}(\nu^km_0,x_0)$ satisfying $u(\overline{x}) \leq \nu^k m_0$.
    Note that $K_{R_2r_2^k}(\nu^km_0,\overline{x}) \subset K_{R_2}(m_0)$ by the choice of $k_2$.
    We then define a second scaled function
    \begin{align*}
        w(x) = \frac{u(\overline{x}+(r_2^k)^{\alpha_i}(\nu^k)^{\beta_i}x)}{ \nu^k}.
    \end{align*}
    Then $w$ is nonnegative and satisfies \eqref{Lepspde} in $K_{R_2}(m_0)$. 
    Since $w(0) \leq m_0$, Corollary \ref{Lepscor} yields
    \begin{align*}
        \left|\{w \geq L_2m_0\} \cap \frac{2}{3}K_1(m_0)\right| \leq \frac{1}{4^{n+1}}\left|\frac{2}{3}K_1(m_0)\right|.
    \end{align*}
    Scaling back to $u$, we get
    \begin{align} \label{harnmeas2}
        \left|\{u \geq \nu^k L_2m_0\} \cap \frac{2}{3}K_{r^k_2}(\nu^km_0,\overline{x})\right| \leq \frac{1}{4^{n+1}}\left|\frac{2}{3}K_{r^k_2}(\nu^km_0)\right|.
    \end{align}
    Note that since $\overline{x} \in K_{r_2^k}(\nu^km_0,x_0)$, we have
    \begin{align*}
        \frac{1}{6}K_{r_2^k}(\nu^km_0,\tilde{x}) \subset  \frac{2}{3}K_{r_2^k}(\nu^km_0,x_0) \cap \frac{2}{3}K_{r_2^k}(\nu^km_0,\overline{x}),
    \end{align*}
    where $\tilde{x} = \frac{x_0+\overline{x}}{2}$.
    Therefore, the estimates \eqref{harnmeas1} and \eqref{harnmeas2} imply 
    \begin{align*}
        \left|\{u \leq \nu^k L_2m_0\} \cap \frac{1}{6}K_{r_2^k}(\nu^km_0,\tilde{x})\right| \leq \frac{1}{4}\left|\frac{1}{6}K_{r_2^k}(\nu^km_0)\right|, \\
        \left|\{u  \geq\nu^k L_2m_0\} \cap \frac{1}{6}K_{r_2^k}(\nu^km_0,\tilde{x})\right| \leq \frac{1}{4}\left|\frac{1}{6}K_{r_2^k}(\nu^km_0)\right|.
    \end{align*}
    Summing the above two inequalities, we  obtain $\left|\frac{1}{6}K_{r^k_2}(\nu^km_0)\right| \leq \frac{1}{2}\left|\frac{1}{6}K_{r^k_2}(\nu^km_0)\right|$.
    This is a contradiction and completes the proof.
\end{proof}

\begin{lem} \label{harnlem}
    Let $u \in C(K_{R_3}(m_0))$ be nonnegative in $K_{R_3}(m_0)$ and satisfy
    \begin{align} \label{harnpde2}
    \begin{cases}
        \mathcal{M}^- \left( \diag{|D_iu|^{\frac{p_i-2}{2}}} D^2u \diag{|D_iu|^{\frac{p_i-2}{2}}} \right) -\mu_0\sum|D_iu|^{p_i-1} \leq  \epsilon_0 \\
        \mathcal{M}^+ \left( \diag{|D_iu|^{\frac{p_i-2}{2}}} D^2u \diag{|D_iu|^{\frac{p_i-2}{2}}} \right) +\mu_0\sum|D_iu|^{p_i-1} \geq  -\epsilon_0 
    \end{cases} 
     \quad \text{ in } K_{R_3}(m_0),
    \end{align}
    for some universal $R_3 > R_2$.
    Then there exists $C_0>1$ such that
    \begin{align*}
        u(0) \leq m_0 \quad &\Longrightarrow \quad \sup_{K_{1/2}(m_0)} u \leq C_0m_0, \\
        u(0) \geq C_0m_0 \quad &\Longrightarrow \quad \inf_{K_{1/2}(m_0)} u \geq m_0.
    \end{align*}
\end{lem}
\begin{proof}
    We begin by choosing $k_3 \geq k_2$ such that $\sum_{k=k_3}^\infty R_2^{\alpha_1}r_2^{k\alpha_n} < 1-\left(\frac{1}{2}\right)^{\alpha_n}$.
    Our purpose is to prove that if $u \geq 0$ in $K_{R_2}(m_0)$ and $u(0) \leq m_0$, then
    \begin{align} \label{harnfirst}
        \sup_{K_{1/2}(m_0)} u < 2 \nu^{k_3-1}L_2m_0 = C_0m_0.
    \end{align}
    We proceed by contradiction.
    Suppose there exists a point $x_0 \in K_{1/2}(m_0)$ such that
    \begin{align*}
        u(x_0) \geq  2 \nu^{k_3-1} L_2m_0.
    \end{align*}
    Then by Lemma \ref{Harnmidlem}, there exists $x_1 \in K_{R_{2}r_2^{k_3}}(\nu^{k_3}m_0,x_0)$ such that
    \begin{align*}
        u(x_1) \geq 2 \nu^{k_3} L_2m_0.
    \end{align*}
    Observe that $x_1 \in K_1(m_0)$ since $\nu^{\beta_i} \leq 1$ and
    \begin{align*}
        |(x_1)_i|\leq m_0^{\beta_i}\left( \frac{1}{2^{\alpha_i}} + R_2^{\alpha_i} r_2^{k_3\alpha_i}  \right) \leq m_0^{\beta_i}.
    \end{align*}
    
    By induction, we claim that there exists a sequence of points $x_{l} \in K_{R_2 r_2^{k_3+l-1}}(\nu^{k_3+l-1}m_0,x_{l-1})$ for $l \geq 1$ satisfying 
    \begin{align} \label{xlcond}
        x_{l} \in K_1(m_0) \quad \text{ and } \quad u(x_l) \geq 2 \nu^{k_3 + l-1} L_2 m_0.
    \end{align} 
    With the base case $l =1$ completed, we prove \eqref{xlcond} for $l+1$ assuming inductively that \eqref{xlcond} holds for every $k \leq l$.
    Using Lemma \ref{Harnmidlem}, there exists $x_{l+1} \in K_{R_2 r_2^{k_3+l}}(\nu^{k_3+l}m_0,x_{l})$ satisfying $u(x_{l+1}) \geq 2 \nu^{k_3 + l} L_2 m_0$.
    Moreover, using $x_{k} \in K_{R_2 r_2^{k_3+k-1}}(\nu^{k_3+k-1}m_0,x_{k-1})$ for any $k \leq l $ and $\nu^{\beta_i} \leq 1$, we obtain
    \begin{align*}
        |(x_{l+1})_i| &\leq |(x_0)_i| + \sum_{k=1}^{l+1}|(x_k)_i-(x_{k-1})_i| \\
        &\leq m_0^{\beta_i}\left( \frac{1}{2^{\alpha_i}} + \sum_{k=k_3}^{k_3+l} R_2^{\alpha_i} r_2^{k\alpha_i}  \right) \leq m_0^{\beta_i}.
    \end{align*}
    This confirms that $x_{l+1} \in K_1(m_0)$, which completes the induction step.
    Therefore, as $l \rightarrow \infty$, $x_l$ converges to some point $x_\infty \in \overline{K_1(m_0)}$.
    However, the continuity of $u$ implies that $u(x_\infty)$ must be finite, which contradicts the fact that $u(x_l) \rightarrow \infty$ as $l \rightarrow \infty$ as required by \eqref{xlcond}.
    This establishes the first part of the lemma with $C_0 = 2 \nu^{k_3-1}L_2>1$.

    The second part is proved by contradiction.
    Let $R_3>R_2$ be chosen such that $R_3^{1/p_n} \geq R_2^{1/p_n}+1$.
    Suppose that $u(0) \geq C_0m_0$, but  there exists $x_0 \in K_{1/2}(m_0)$ such that $u(x_0) < m_0$.
    Since $K_{R_2}(m_0,x_0) \subset K_{R_3}(m_0)$ by the choice of $R_3$, the previous argument \eqref{harnfirst} yields $\sup_{K_{1/2}(m_0,x_0)} u < C_0m_0$.
    However, this contradicts our initial assumption $u(0) \geq C_0m_0$, which completes the proof.
\end{proof}

With the preceding lemma at hand, we are ready to prove the main theorem, Theorem \ref{Mainthm}.

\begin{proof}[Proof of Theorem \ref{Mainthm}]
    We set $C_0 = 2 \nu^{k_3-1}L_2$ as in Lemma \ref{harnlem} and $R_0 = 2R_3$.
    Assume that $u$ is defined in $K_{R_0r}(u(0)) \subset \Omega$.
    We define a rescaled function $v \in C(K_{R_3}(m_0))$ as
    \begin{align*}
        v(x) = \frac{u((2r)^{\alpha_i}M_1^{\beta_i}x)}{M_1},
    \end{align*}
    where $M_1 = \frac{u(0)}{m_0}>0$.
    Then  $v$ is nonnegative in $K_{R_3}(m_0)$, and satisfies
    \begin{align} \label{harnpdescal}
    \begin{cases}
        \mathcal{M}^- \left( \diag{|D_iv|^{\frac{p_i-2}{2}}} D^2v \diag{|D_iv|^{\frac{p_i-2}{2}}} \right) -\mu\sum_i \frac{(2r)^{1/p_i}}{M^{(p_n-p_i)/p_i}}|D_iv|^{p_i-1} \leq \frac{2r}{M^{p_n-1}}c_0  \\
        \mathcal{M}^+ \left( \diag{|D_iv|^{\frac{p_i-2}{2}}} D^2v \diag{|D_iv|^{\frac{p_i-2}{2}}} \right) +\mu\sum_i \frac{(2r)^{1/p_i}}{M^{(p_n-p_i)/p_i}}|D_iv|^{p_i-1} \geq -\frac{2r}{M^{p_n-1}}c_0
    \end{cases}
     \text{ in }K_{R_3}(m_0)
    \end{align}
    with $M=M_1$.
    Also, we have $v(0) \leq m_0$.
    By selecting $\epsilon_1 =  \frac{\epsilon_0}{2m_0^{p_n-1}}$, we have $\frac{2}{M_1^{p_n-1}}c_0 \leq \epsilon_0$.
    Moreover, we choose small enough $\mu_1>0$ such that $\mu_1 \frac{2^{1/p_i}}{M_1^{(p_n-p_i)/p_i}}\leq \mu_0$ for any $i$.
    Then $v$ satisfies \eqref{harnpde2}, so we obtain
    \begin{align*}
        \sup_{K_{1/2}(m_0)} v \leq C_0m_0,
    \end{align*}
    by Lemma \ref{harnlem}.
    Scaling back to $u$, we get
    \begin{align*}
        \sup_{K_{r}(u(0))} u \leq C_0u(0).
    \end{align*}
    This completes the first part of the theorem.
    
    For the second part, assume that $u$ is defined in $K_{R_0r}(u(0)/C_0) \subset \Omega$.
    We define a second rescaled function $w \in C(K_{R_3}(m_0))$ as
    \begin{align*}
        w(x) = \frac{u((2r)^{\alpha_i}M_2^{\beta_i}x)}{M_2},
    \end{align*}
    where $M_2 = \frac{u(0)}{C_0m_0}>0$.
    Then  $w$ is nonnegative in $K_{R_3}(m_0)$, and satisfies \eqref{harnpdescal} with $M=M_2$.
    Moreover, we have $w(0) \geq C_0m_0$.
    We choose $\epsilon_2 = \frac{\epsilon_0}{2(C_0m_0)^{p_n-1}}$ and sufficiently small $\mu_2 >0$ satisfying $\mu_2 \frac{2^{1/p_i}}{M_2^{(p_n-p_i)/p_i}}\leq \mu_0$ for any $i$.
    Then $w$ satisfies \eqref{harnpde2}, so we obtain
    \begin{align*}
        \inf_{K_{1/2}(m_0)} w \geq m_0,
    \end{align*}
    by Lemma \ref{harnlem}.
    Scaling back to $u$, we get
    \begin{align*}
        \inf_{K_{r}(u(0)/C_0)} u \geq \frac{1}{C_0}u(0).
    \end{align*}
    This completes the proof of Theorem \ref{Mainthm}.
\end{proof}

\section{Sharpness in the Constant-Coefficient Case}\label{sec8}
We now prove the sharpness of the exponent condition \eqref{pcondLap} in the constant-coefficient case $\lambda=\Lambda$.
The counterexamples have different forms in the supercritical and critical regimes.
Above the critical threshold, we construct a nontrivial nonnegative entire solution that vanishes on the hyperplane $\{x_n=0\}$.
At the critical threshold, we instead construct a family of strictly positive entire solutions whose Harnack ratios become arbitrarily large.
We therefore treat the two cases separately.

We conclude the section with a fully nonlinear example showing that the constant-coefficient exponent range cannot be retained for arbitrary ellipticity ratios.

\begin{proof}[Proof of Theorem \ref{Mainthm2}] \

    \medskip
    \noindent\textit{Supercritical case.} Assume that \eqref{psupcrit} holds.
    Our construction is motivated by the barrier in Section \ref{sec4} at the limiting exponent $q=p_n-1$.
    We reverse the sign of the $x_n$-term and remove the term $t$, so that the resulting function can vanish on $\{x_n=0\}$.
    We also introduce a convex term $\psi(t) \geq0$ to keep the maximum finite when all transverse terms vanish.
    More precisely, we define
    \begin{align} \label{Wdef}
        W(x) :=  \max_{t\geq 0} \left\{ G(x,t)\right\} := \max_{t\geq 0} \left\{d_n |x_n|^{\frac{p_n}{p_n-1}}t- \psi(t) - \sum_{i<n} d_i |x_i|^{\frac{p_i}{p_i-1}}t^{\frac{p_n-1}{p_i-1}}\right\}.
    \end{align}
    Here $d_i \geq0$ for $i<n$, $d_n>0$, and $\psi(t)$ will be determined later.
    We temporarily set $d_i=0$ for $i<n$ whenever $p_i=p_n$, since the corresponding power of $t$ would be linear and could prevent the existence of a positive maximizer.
    Indeed, in this case, we will set $d_i=0$ below.
    We choose $\psi \geq0$ so that
    \begin{align*}
        \psi(0)=\psi'(0)=0, \quad \psi''(t)>0, \quad \lim_{t\rightarrow \infty} \frac{\psi(t)}{t}=\infty.
    \end{align*}
    If $x_n=0$, then $G(x,t) \leq 0=G(x,0)$, and hence $W(x)=0$.
    On the other hand, if $x_n\neq0$, since  $\psi(t)=o(t)$ as $t\rightarrow0$ and $\frac{p_n-1}{p_i-1}>1$, we have $G(x,t)=d_n|x_n|^{\frac{p_n}{p_n-1}}t+o(t)>0$ for sufficiently small $t>0$, whereas $\lim_{t\rightarrow \infty} \frac{\psi(t)}{t}=\infty$ implies $\lim_{t\rightarrow \infty} G(x,t)=-\infty$.
    Moreover, $\psi''>0$ makes $G(x,t)$ strictly concave in $t$. 
    Therefore, for every $x \in \{x_n\neq0\}$, there exists a unique maximizer $T(x)>0$.
    Consequently,
    \begin{align*}
        W =0 \quad \text{ on } \{x_n = 0\}, \quad W >0 \text{ on } \{x_n \neq 0\}.
    \end{align*}
    Moreover, the maximizing condition on $t=T(x)$ gives
    \begin{align}\label{DtG}
        D_tG(x,T(x))=0.
    \end{align}
    By the implicit function theorem and $D_{tt}G<0$, we obtain $T \in C^1$, and consequently $W \in C^1(\{x_n\neq0\})$.
    Furthermore, for $i<n$, we have
    \begin{align} \label{DiW}
        D_iW(x) = D_iG(x,T) = -\frac{d_ip_i}{p_i-1}\sgn{x_i}|x_i|^{\frac{1}{p_i-1}} T^{\frac{p_n-1}{p_i-1}}.
    \end{align}
    If $p_i=2$, then $D_iW(x) = -2d_ix_iT^{p_n-1} \in C^1$.
    If $p_j>2$ and $d_j>0$, however, $D_{jj}W$ may become singular on $\{x_j=0\}$.
    We therefore perform the computation away from these coordinate hyperplanes, $$\Sigma:=\bigcup_{\{j:j<n,p_j>2,d_j>0\}}\{x_j=0\} \quad  \text{ and } \quad \{x_n=0\}.$$

    For $x \in \R^n \setminus (\Sigma \cup \{x_n=0\})$,
    as in Section \ref{sec4}, differentiating \eqref{DtG} and $W(x)=G(x,T(x))$ gives
    \begin{align*}
        D_iT = -\frac{D_{it}G}{D_{tt}G}, \quad D_iW = D_iG, \quad D_{ii}W = D_{ii}G - \frac{(D_{it}G)^2}{D_{tt}G},
    \end{align*}
    where all derivatives are evaluated at $(x,T(x))$.
    Consequently,
    \begin{align*}
        \Delta_{(p_i)}W  = \sum_i|D_iG|^{p_i-2}D_{ii}G + \frac{\sum_i|D_iG|^{p_i-2}|D_{it}G|^2}{-D_{tt}G}.
    \end{align*}
    Thus, following the coefficient-matching argument in Section \ref{sec4}, we seek the following identities:
    \begin{align}
        \sum_i|D_iG|^{p_i-2}D_{ii}G &= -T^{p_n-1}, \label{DGDDiiG}\\
        \sum_i|D_iG|^{p_i-2}|D_{it}G|^2&= -T^{p_n-1}D_{tt}G. \label{DGDDitG}
    \end{align}
    For $i< n$,  direct differentiation yields
    \begin{align} \label{DiGDiiG2}
        |D_iG|^{p_i-2}D_{ii}G = -\frac{(d_ip_i)^{p_i-1}}{(p_i-1)^{p_i}}T^{p_n-1}, \quad |D_nG|^{p_n-2}D_{nn}G = \frac{(d_np_n)^{p_n-1}}{(p_n-1)^{p_n}}T^{p_n-1}=\kappa T^{p_n-1},
    \end{align}
    where $\kappa :=\frac{(d_np_n)^{p_n-1}}{(p_n-1)^{p_n}}>0$.
    Then the first identity \eqref{DGDDiiG} requires
    \begin{align} \label{dikappacond}
        \sum_{i<n} \frac{(d_ip_i)^{p_i-1}}{(p_i-1)^{p_i}}=1+\kappa.
    \end{align}
    Furthermore, for $i<n$,  direct differentiation yields
    \begin{align*}
        |D_iG|^{p_i-2}|D_{it}G|^2 = T^{p_n-2}\frac{(d_ip_i)^{p_i}(p_n-1)^2}{(p_i-1)^{p_i+2}} |x_i|^{\frac{p_i}{p_i-1}} T^{\frac{p_n-p_i}{p_i-1}}.
    \end{align*}
    In the $n$-th direction, we obtain
    \begin{align} \label{DnGDntG}
        |D_nG|^{p_n-2}|D_{nt}G|^2 = p_n\kappa d_n|x_n|^{\frac{p_n}{p_n-1}} T^{p_n-2}.
    \end{align}
    On the other hand,
    \begin{align} \label{TDttG}
        -TD_{tt}G = T\psi''(T) + \sum_{i<n} \frac{(p_n-1)(p_n-p_i)}{(p_i-1)^2} d_i|x_i|^{\frac{p_i}{p_i-1}}T^{\frac{p_n-p_i}{p_i-1}}.
    \end{align}
    Observe that by \eqref{DtG}, we have
    \begin{align} \label{DtGTequ}
        d_n|x_n|^{\frac{p_n}{p_n-1}} = \psi'(T) +\sum_{i<n} \frac{p_n-1}{p_i-1} d_i |x_i|^{\frac{p_i}{p_i-1}}T^{\frac{p_n-p_i}{p_i-1}}.
    \end{align}
    Substituting the above expression into \eqref{DnGDntG}, we obtain
    \begin{align*}
        \frac{1}{T^{p_n-2}}\sum_i|D_iG|^{p_i-2}|D_{it}G|^2 =  p_n\kappa\psi'(T) + \sum_{i<n} \left(\frac{(d_ip_i)^{p_i}(p_n-1)^2}{(p_i-1)^{p_i+2}} + \frac{p_n\kappa(p_n-1)d_i}{p_i-1}\right)|x_i|^{\frac{p_i}{p_i-1}}T^{\frac{p_n-p_i}{p_i-1}}.
    \end{align*}
    Matching the coefficients above with \eqref{TDttG}, we need the following differential equation
    \begin{align} \label{diffequ}
         t\psi''(t) = p_n \kappa \psi'(t), \quad \psi(0)=0,
    \end{align}
    and for $d_i \geq 0$ with $i<n$,
    \begin{align} \label{diequ}
        \frac{(p_n-1)(p_n-p_i)}{(p_i-1)^2} d_i=\frac{(d_ip_i)^{p_i}(p_n-1)^2}{(p_i-1)^{p_i+2}} + \frac{p_n\kappa(p_n-1)d_i}{p_i-1}.
    \end{align}
    We choose the particular solution $\psi$ of \eqref{diffequ} as
    \begin{align} \label{psidef}
        \psi(t) = \frac{t^{1+p_n\kappa}}{1+p_n\kappa}, 
    \end{align}
    and choose $d_i \geq 0$ with $i<n$ as
    \begin{align} \label{didef}
        d_i = \frac{(p_i-1)^{\frac{p_i}{p_i-1}}}{p_i} \left( \frac{(p_n-p_i-p_n\kappa(p_i-1))_+}{p_i(p_n-1)}\right)^{\frac{1}{p_i-1}}.
    \end{align}
    Note that when the numerator is nonpositive, setting $d_i=0$ makes both sides of \eqref{diequ} vanish. 
    Moreover, if $p_i=p_n$ for some $i$, then the numerator becomes nonpositive, so $d_i=0$.
    Observe that with this choice of $d_i$,
    \begin{align} \label{dipiequ}
        \frac{(d_ip_i)^{p_i-1}}{(p_i-1)^{p_i}} = \frac{(p_n-p_i-p_n\kappa(p_i-1))_+}{p_i(p_n-1)}.
    \end{align}
    Thus, the second identity \eqref{DGDDitG} holds under this choice.
    
    It remains to choose $\kappa>0$ so that the first identity \eqref{DGDDiiG} also holds.
    Using \eqref{dipiequ}, the condition \eqref{dikappacond} becomes
    \begin{align} \label{kappacond}
        \sum_{i<n} \frac{(p_n-p_i-p_n\kappa(p_i-1))_+}{p_i(p_n-1)} = 1+\kappa.
    \end{align}
    Observe that the condition \eqref{psupcrit} is equivalent to
    \begin{align*}
        \sum_{i<n} \frac{p_n-p_i}{p_i(p_n-1)} >1.
    \end{align*}
    Therefore, the left-hand side of \eqref{kappacond} is greater than $1$ when $\kappa=0$ and nonincreasing in $\kappa$, and vanishes for sufficiently large $\kappa$.
    On the other hand, the right-hand side is strictly increasing from $1$.
    Hence there exists a unique $\kappa>0$ satisfying this equation.
    Then we have
    \begin{align*}
        d_n=\frac{(p_n-1)^{\frac{p_n}{p_n-1}}}{p_n}\kappa^{\frac{1}{p_n-1}}.
    \end{align*}
    With this choice of coefficients, the first identity \eqref{DGDDiiG} holds.
    Combining the two identities, we conclude that
    \begin{align*}
        \Delta_{(p_i)}W = 0,
    \end{align*}
    away from the singular coordinate hyperplanes, $\Sigma$ and $\{x_n=0\}$.

    It remains to verify the equation at the points where $W$ may fail to be $C^2$.
    Using \eqref{psidef}, we have
    \begin{align*}
        0 \leq W(x) \leq \max_{t\geq0}\left\{d_n|x_n|^{\frac{p_n}{p_n-1}}t-\frac{t^{1+p_n\kappa}}{1+p_n\kappa}\right\} =\frac{p_n\kappa}{1+p_n\kappa}\left(d_n|x_n|^{\frac{p_n}{p_n-1}}\right)^{\frac{1+p_n\kappa}{p_n\kappa}} =o(|x_n|).
    \end{align*}
    Therefore, $W$ is differentiable on $\{x_n=0\}$ and $DW=0$ there.
    Furthermore, since \eqref{DtGTequ} implies $T(x) \rightarrow 0$ as $|x_n|\rightarrow0$, and \eqref{DiW},
    we have $DW(x) \rightarrow 0$ as $|x_n|\rightarrow0$.
    Consequently, $W \in C^1(\R^n)$.

    Let $y \in \Sigma$ with $y_n\neq0$, and set $J=\{ j:j<n,y_j=0,p_j>2,d_j>0\}$.
    Then for $j\in J$, we have
    \begin{align*}
        W(y+he_j)-W(y)=-d_jT(y)^{\frac{p_n-1}{p_j-1}}|h|^{\frac{p_j}{p_j-1}}+o\left(|h|^{\frac{p_j}{p_j-1}}\right).
    \end{align*}
    Since $1<\frac{p_j}{p_j-1}<2$, no $C^2$ function can touch $W$ from below at $y$.
    Suppose that $\phi \in C^2$ touches $W$ from above at $y$. 
    Then since $D_j\phi(y)=D_jW(y)=0$ for $j \in J$, the $j$-th derivative terms in the equation vanish.
    Moreover, on the slice $S=\{z : z_j=0 \text{ for } j \in J\}$ $W$ is $C^2$, and $D^2 \phi|_S(y) \geq D^2 W|_S(y)$.
    Therefore,
    \begin{align*}
    \Delta_{(p_i)}\phi(y)&=\sum_{i\notin J}|D_i\phi|^{p_i-2}D_{ii}\phi(y)\geq \sum_{i\notin J}|D_iW|^{p_i-2}D_{ii}W(y)\\
    &=T^{p_n-1} \sum_{j\in J}  \frac{(d_jp_j)^{p_j-1}}{(p_j-1)^{p_j}}\geq0.
    \end{align*}
    Thus, the equation holds in the viscosity sense at $y$.

    Finally, let $y_n=0$, and suppose that $\phi \in C^2$ touches $W$ from above (respectively, below) at $y$.
    Then we have $D\phi(y)= DW(y)=0$, so all terms with $p_i>2$ vanish.
    Since every $p_i=2$ direction is tangential to $\{x_n=0\}$, we obtain $D_{ii}\phi(y) \geq 0 \text{ (respectively, }D_{ii}\phi(y)\leq0)$ whenever $p_i=2$.
    Therefore, $\Delta_{(p_i)}\phi(y)=\sum_{\{i:p_i=2\}}D_{ii}\phi \geq0$ (respectively, $\Delta_{(p_i)}\phi(y)\leq0$).
    Consequently, $W$ is a viscosity solution of $\Delta_{(p_i)}W =0$ in $\R^n$.
    This proves part (1) of the theorem.

    \medskip
\noindent\textit{Critical case.}
    Assume that the critical condition \eqref{pcrit} holds, equivalently,
    \begin{align*}
        \sum_{i<n} \frac{p_n-p_i}{p_i(p_n-1)} =1.
    \end{align*}
    In this case, the condition for $\kappa$ in \eqref{kappacond} has only the solution $\kappa=0$.
    This gives $d_n=0$, and the preceding construction yields only the zero function.
    We therefore introduce the perturbation term $\epsilon x_nt$ with $\epsilon>0$ and replace the power term by the logarithmic term $t\log{t}-t$.
    For $i<n$, we retain the definition of $d_i$ in \eqref{didef} with $\kappa=0$, namely,
    \begin{align} \label{didef2}
        d_i = \frac{(p_i-1)^{\frac{p_i}{p_i-1}}}{p_i} \left( \frac{p_n-p_i}{p_i(p_n-1)}\right)^{\frac{1}{p_i-1}}.
    \end{align}
    Then, 
    \begin{align} \label{dipisum}
        \frac{(p_n-1)(p_n-p_i)}{(p_i-1)^2} d_i=\frac{(d_ip_i)^{p_i}(p_n-1)^2}{(p_i-1)^{p_i+2}}, \quad \sum_{i<n}\frac{(d_ip_i)^{p_i-1}}{(p_i-1)^{p_i}} =1.
    \end{align}
    For $\epsilon>0$, we now define
    \begin{align} \label{Wedef}
        W_\epsilon(x',x_n) :=  \max_{t\geq 0} \left\{ G_\epsilon(x,t)\right\} := \max_{t\geq 0} \left\{\epsilon x_nt- \epsilon^{p_n}(t\log{t}-t) - \sum_{i<n} d_i |x_i|^{\frac{p_i}{p_i-1}}t^{\frac{p_n-1}{p_i-1}}\right\}.
    \end{align}
    For $t>0$, direct differentiation gives
    \begin{align*}
        D_tG_{\epsilon} &= \epsilon x_n -\epsilon^{p_n} \log{t} - \sum_{i<n} \frac{p_n-1}{p_i-1}d_i  |x_i|^{\frac{p_i}{p_i-1}}t^{\frac{p_n-p_i}{p_i-1}}.
    \end{align*}
    Thus, $G_\epsilon(x,\cdot)$ is strictly concave, $\lim_{t\rightarrow0+}D_tG_{\epsilon}=+\infty$ and $\lim_{t\rightarrow\infty}D_tG_{\epsilon}=-\infty$.
    Therefore, for every $x\in\R^n$, the maximum is attained at a unique point $T(x)$.
    By the implicit function theorem, we have $T \in C^1$, which implies  $W \in C^1(\R^n)$.
    
    We now show that $W_\epsilon>0$.
    Using $D_tG_\epsilon(x,T)=0$, we get
    \begin{align} \label{Tequ}
        \epsilon x_n T = \epsilon^{p_n} T\log{T} + \sum_{i<n} \frac{p_n-1}{p_i-1}d_i  |x_i|^{\frac{p_i}{p_i-1}}T^{\frac{p_n-1}{p_i-1}}.
    \end{align}
    Substituting this into the definition of $W_\epsilon$, we have
    \begin{align*}
        W_\epsilon = \epsilon^{p_n}T + \sum_{i<n} \frac{p_n-p_i}{p_i-1}d_i |x_i|^{\frac{p_i}{p_i-1}}T^{\frac{p_n-1}{p_i-1}} >0.
    \end{align*}
    We next verify the equation away from the singular coordinate hyperplanes $$\Sigma=\bigcup_{\{j:j<n,p_j>2,d_j>0\}}\{x_j=0\}.$$
    As in the preceding case, for $x \in \R^n \setminus \Sigma$ we obtain
    \begin{align*}
        \Delta_{(p_i)}W_\epsilon  = \sum_i|D_iG_\epsilon|^{p_i-2}D_{ii}G_\epsilon + \frac{\sum_i|D_iG_\epsilon|^{p_i-2}|D_{it}G_\epsilon|^2}{-D_{tt}G_\epsilon},
    \end{align*}
    where all derivatives are evaluated at $(x,T(x))$.
    Our purpose is to show the following identities:
    \begin{align} \label{critaim}
        \sum_i|D_iG_\epsilon|^{p_i-2}D_{ii}G_\epsilon = -T^{p_n-1}, \quad 
        \sum_i|D_iG_\epsilon|^{p_i-2}|D_{it}G_\epsilon|^2= -T^{p_n-1}D_{tt}G_\epsilon.
    \end{align}
    For $i<n$, the same calculation as in the preceding case yields
    \begin{align*}
        |D_iG_\epsilon|^{p_i-2}D_{ii}G_\epsilon &= -\frac{(d_ip_i)^{p_i-1}}{(p_i-1)^{p_i}}T^{p_n-1},\\ |D_iG_\epsilon|^{p_i-2}|D_{it}G_\epsilon|^2 &= T^{p_n-2}\frac{(d_ip_i)^{p_i}(p_n-1)^2}{(p_i-1)^{p_i+2}} |x_i|^{\frac{p_i}{p_i-1}} T^{\frac{p_n-p_i}{p_i-1}}.
    \end{align*}
    In the last direction,
    \begin{align*}
        D_nG_\epsilon= \epsilon T, \quad D_{nn}G_\epsilon= 0, \quad D_{nt}G_\epsilon= \epsilon.
    \end{align*}
    Therefore, we have
    \begin{align*}
        |D_nG_\epsilon|^{p_n-2}D_{nn}G_\epsilon = 0, \quad |D_nG_\epsilon|^{p_n-2}|D_{nt}G_\epsilon|^2 = \epsilon^{p_n} T^{p_n-2}.
    \end{align*}
    Thus, \eqref{dipisum} gives
    \begin{align*}
        \sum_i|D_iG_\epsilon|^{p_i-2}D_{ii}G_\epsilon = \sum_{i<n}-\frac{(d_ip_i)^{p_i-1}}{(p_i-1)^{p_i}}T^{p_n-1}=-T^{p_n-1},
    \end{align*}
    so the first identity of \eqref{critaim} holds.
    For the second identity, we have
    \begin{align*}
        -T D_{tt}G_\epsilon(x,T) = \epsilon^{p_n} +\sum_{i<n} \frac{(p_n-1)(p_n-p_i)}{(p_i-1)^2} d_i|x_i|^{\frac{p_i}{p_i-1}}T^{\frac{p_n-p_i}{p_i-1}},
    \end{align*}
    and
    \begin{align*}
        \frac{1}{T^{p_n-2}}\sum_i|D_iG_\epsilon|^{p_i-2}|D_{it}G_\epsilon|^2 = \epsilon^{p_n} + \sum_{i<n}\frac{(d_ip_i)^{p_i}(p_n-1)^2}{(p_i-1)^{p_i+2}} |x_i|^{\frac{p_i}{p_i-1}} T^{\frac{p_n-p_i}{p_i-1}}.
    \end{align*}
    Therefore, using \eqref{dipisum} yields the second identity of \eqref{critaim}.
    Combining the two identities, we conclude that
    \begin{align*}
        \Delta_{(p_i)}W_\epsilon =0,
    \end{align*}
    away from the singular coordinate hyperplanes $\Sigma$.
    The viscosity verification on $\Sigma$ is the same as in the preceding case.
    
    Finally, on the $x_n$-axis, the equation \eqref{Tequ} reduces to $\epsilon x_n = \epsilon^{p_n} \log T$.
    Therefore,
    \begin{align*}
        T(0,x_n) = \exp\left(\frac{x_n}{\epsilon^{p_n-1}}  \right), \quad W_\epsilon(0,x_n)=\epsilon^{p_n}\exp\left(\frac{x_n}{\epsilon^{p_n-1}}  \right).
    \end{align*}
    In particular, $W_\epsilon(0) = \epsilon^{p_n}$.
    Since the width of $K_1(M)$ in the $x_n$-direction is independent of $M$, we have $ e_n/2 \in K_1(W_\epsilon(0))$.
    Therefore,
    \begin{align*}
        \frac{\sup_{K_1(W_\epsilon(0))}W_\epsilon}{W_\epsilon(0)} \geq \frac{W_\epsilon(e_n/2)}{W_\epsilon(0)} = \exp\left(\frac{1}{2\epsilon^{p_n-1}}\right) \longrightarrow \infty \quad \text{ as } \epsilon \rightarrow 0.
    \end{align*}
    This proves part (2) of the theorem.
\end{proof}

Finally, we show that the constant-coefficient condition \eqref{pcondLap} cannot hold uniformly with respect to the ellipticity ratio. 
In particular, some dependence of the admissible exponent range on $\lambda/\Lambda$ is unavoidable.

\begin{prop} \label{lamprop}
    Let $n=3$, $(p_1,p_2,p_3)=(3,3,6)$ and assume $\frac{\lambda}{\Lambda}\leq \frac{1}{4}$.
    Note that 
    \begin{align*}
        \bar{p} =\frac{18}{5}>3=n.
    \end{align*}
    Then there exists a family $W_\epsilon \in C^1(\R^3)$ of positive functions satisfying
    \begin{align} \label{Wequ}
        \mathcal{M}^-_{\lambda,\Lambda} \left( \diag{|D_iW_\epsilon|^{\frac{p_i-2}{2}}} D^2W_\epsilon \diag{|D_iW_\epsilon|^{\frac{p_i-2}{2}}} \right) \leq 0 \leq \mathcal{M}^+_{\lambda,\Lambda} \left( \diag{|D_iW_\epsilon|^{\frac{p_i-2}{2}}} D^2W_\epsilon \diag{|D_iW_\epsilon|^{\frac{p_i-2}{2}}} \right) \quad \text{ in }\R^3
    \end{align}
    in the viscosity sense, and
        \begin{align*}
            \frac{\sup_{K_1(W_\epsilon(0))}W_\epsilon}{W_\epsilon(0)} \longrightarrow \infty, \quad \text{ as }\epsilon\rightarrow 0. 
        \end{align*}
    Consequently, the Harnack inequality for the fully nonlinear class can fail even when $\bar p>n$.
\end{prop}

\begin{proof}
    We use the same max-type construction as in the critical case \eqref{Wedef}, now with $(p_1,p_2,p_3)=(3,3,6)$.
    For $\epsilon>0$, define
    \begin{align} \label{Wedef2}
        W_\epsilon(x,y,z) :=  \max_{t\geq 0} \left\{ G_\epsilon(x,y,z,t)\right\} := \max_{t\geq 0} \left\{\epsilon zt- \epsilon^{6}(t\log{t}-t) -  d \left( |x|^{\frac{3}{2}}+|y|^{\frac{3}{2}}\right) t^{\frac{5}{2}}\right\}
    \end{align}
    where $d=\frac{2\sqrt{10}}{15}$ is precisely the coefficient $d_i$ in \eqref{didef2}.
    As in the critical case, the maximum is attained at a unique point $T(x)>0$.
    Moreover, $W_\epsilon>0$ and $W_\epsilon \in C^1(\R^3)$.
    We first work on $\{xy\neq0\}$.
    The same differentiation as in the preceding construction gives
    \begin{align*}
        |D_iG_\epsilon|^{p_i-2}D_{ii}G_\epsilon &= -\frac{1}{5}T^{5} \quad (i=1,2), \quad D_{33}G_\epsilon=0,
    \end{align*}
    Moreover, we obtain
    \begin{align} \label{vsize}
        \sum_{i=1}^{3}|D_iG_\epsilon|^{p_i-2}|D_{it}G_\epsilon|^2= -T^{5}D_{tt}G_\epsilon,
    \end{align}
    where all derivatives are evaluated at $(x,T(x))$.
    Consequently,
    \begin{align*}
        B:=\diag{|D_iW_\epsilon|^{\frac{p_i-2}{2}}} D^2W_\epsilon \diag{|D_iW_\epsilon|^{\frac{p_i-2}{2}}}= v \otimes v -\frac{T^5}{5} \diag{1,1,0}, \quad v=\left(\frac{|D_iG_\epsilon|^{\frac{p_i-2}{2}}D_{it}G_\epsilon}{\sqrt{-D_{tt}G_\epsilon}} \right).
    \end{align*}
    By \eqref{vsize},
    \begin{align*}
        |v|^2 = T^5, \quad \tr{B}=\frac{3}{5}T^5.
    \end{align*}
    Observe that $B$ has an eigenvalue $-T^5/5$, and the corresponding eigenvector is $\xi=(\xi_1,\xi_2,0)$ orthogonal to $v$.
    Therefore, 
    \begin{align*}
        \mathcal{M}^-_{\lambda,\Lambda}(B)  &= \lambda\tr{B} + (\Lambda-\lambda) \sum_{e_j(B)<0} e_j(B)\\
        & \leq \frac{3}{5}\lambda T^5 - \frac{\Lambda-\lambda}{5}T^5 \\
        &= \frac{4\lambda-\Lambda}{5}T^5\leq0,
    \end{align*}
    whereas
    \begin{align*}
        \mathcal{M}^+_{\lambda,\Lambda}(B)  \geq \lambda\tr{B}=\frac{3}{5}\lambda T^5\geq0.
    \end{align*}
    Therefore, \eqref{Wequ} holds classically on $\{xy\neq0\}$.
    Although $W_\epsilon$ may fail to be $C^2$ on the coordinate hyperplanes $\{x=0\} \cup \{y=0\}$, the same slice argument as in the proof of Theorem \ref{Mainthm2} shows that \eqref{Wequ} holds there in the viscosity sense.
    Thus, \eqref{Wequ} holds in $\R^3$.

    Finally, on the $z$-axis, we have
    \begin{align*}
        W_\epsilon(0,0,z)=\epsilon^{6}\exp\left(\frac{z}{\epsilon^{5}}  \right).
    \end{align*}
    Since $W_\epsilon(0)= \epsilon^6$ and $e_3/2 \in K_1(W_\epsilon(0))$, we obtain
    \begin{align*}
        \frac{\sup_{K_1(W_\epsilon(0))}W_\epsilon}{W_\epsilon(0)} \geq \exp\left(\frac{1}{2\epsilon^{5}}\right) \longrightarrow \infty \quad \text{ as } \epsilon \rightarrow 0.
    \end{align*}
    This completes the proof.
\end{proof}

\textbf{Data Availability}
Data sharing is not applicable to this article, as no datasets were generated or analyzed during the current study.
	
\textbf{Conflict of interest}
The authors declared that they have no conflict of interest.
	
\bibliographystyle{amsplain}
\bibliography{ref}

@article {Bousquet20,
    AUTHOR = {Bousquet, P. and Brasco, L.},
     TITLE = {Lipschitz regularity for orthotropic functionals with
              nonstandard growth conditions},
   JOURNAL = {Rev. Mat. Iberoam.},
  FJOURNAL = {Revista Matem\'atica Iberoamericana},
    VOLUME = {36},
      YEAR = {2020},
    NUMBER = {7},
     PAGES = {1989--2032},
      ISSN = {0213-2230,2235-0616},
   MRCLASS = {35J62 (35B65 35J70 49K20 49N60)},
  MRNUMBER = {4163990},
       DOI = {10.4171/rmi/1189},
       URL = {https://doi.org/10.4171/rmi/1189},
}

@article {Demengel17,
    AUTHOR = {Demengel, F.},
     TITLE = {Regularity properties of viscosity solutions for fully
              nonlinear equations on the model of the anisotropic
              {$\vec p$}-{L}aplacian},
   JOURNAL = {Asymptot. Anal.},
  FJOURNAL = {Asymptotic Analysis},
    VOLUME = {105},
      YEAR = {2017},
    NUMBER = {1-2},
     PAGES = {27--43},
      ISSN = {0921-7134,1875-8576},
   MRCLASS = {35J62 (35B65 35D40)},
  MRNUMBER = {3715607},
MRREVIEWER = {Maria\ Stella\ Fanciullo},
       DOI = {10.3233/asy-171433},
       URL = {https://doi.org/10.3233/asy-171433},
}

@article {Demengel16,
    AUTHOR = {Birindelli, I. and Demengel, F.},
     TITLE = {Existence and regularity results for fully nonlinear operators
              on the model of the pseudo {P}ucci's operators},
   JOURNAL = {J. Elliptic Parabol. Equ.},
  FJOURNAL = {Journal of Elliptic and Parabolic Equations},
    VOLUME = {2},
      YEAR = {2016},
    NUMBER = {1-2},
     PAGES = {171--187},
      ISSN = {2296-9020,2296-9039},
   MRCLASS = {35J60 (35A01 35B50 35B65 35D40 35J92)},
  MRNUMBER = {3645943},
MRREVIEWER = {Barbara\ Brandolini},
       DOI = {10.1007/BF03377400},
       URL = {https://doi.org/10.1007/BF03377400},
}

@article {Bousquet18,
    AUTHOR = {Bousquet, P. and Brasco, L. and Leone, C. and Verde, A.},
     TITLE = {On the {L}ipschitz character of orthotropic {$p$}-harmonic
              functions},
   JOURNAL = {Calc. Var. Partial Differential Equations},
  FJOURNAL = {Calculus of Variations and Partial Differential Equations},
    VOLUME = {57},
      YEAR = {2018},
    NUMBER = {3},
     PAGES = {Paper No. 88, 33},
      ISSN = {0944-2669,1432-0835},
   MRCLASS = {35J62 (35B65 35J70 49K20)},
  MRNUMBER = {3798025},
MRREVIEWER = {Teemu\ Lukkari},
       DOI = {10.1007/s00526-018-1349-3},
       URL = {https://doi.org/10.1007/s00526-018-1349-3},
}

@article {Bousquet16,
    AUTHOR = {Bousquet, P. and Brasco, L. and Julin, V.},
     TITLE = {Lipschitz regularity for local minimizers of some widely
              degenerate problems},
   JOURNAL = {Ann. Sc. Norm. Super. Pisa Cl. Sci. (5)},
  FJOURNAL = {Annali della Scuola Normale Superiore di Pisa. Classe di
              Scienze. Serie V},
    VOLUME = {16},
      YEAR = {2016},
    NUMBER = {4},
     PAGES = {1235--1274},
      ISSN = {0391-173X,2036-2145},
   MRCLASS = {35J62 (35B65 35J70 49N60)},
  MRNUMBER = {3616334},
MRREVIEWER = {Gabriella\ Zecca},
}

@article {Demengel162,
    AUTHOR = {Demengel, F.},
     TITLE = {Lipschitz interior regularity for the viscosity and weak
              solutions of the pseudo {$p$}-{L}aplacian equation},
   JOURNAL = {Adv. Differential Equations},
  FJOURNAL = {Advances in Differential Equations},
    VOLUME = {21},
      YEAR = {2016},
    NUMBER = {3-4},
     PAGES = {373--400},
      ISSN = {1079-9389},
   MRCLASS = {35J62 (35B51 35B65 35D30 35D40 35J92)},
  MRNUMBER = {3461298},
       URL = {http://projecteuclid.org/euclid.ade/1455805262},
}

@book {Caffarelli95,
    AUTHOR = {Caffarelli, L. A. and Cabr\'{e}, X.},
     TITLE = {Fully nonlinear elliptic equations},
    SERIES = {American Mathematical Society Colloquium Publications},
    VOLUME = {43},
 PUBLISHER = {American Mathematical Society, Providence, RI},
      YEAR = {1995},
     PAGES = {vi+104},
      ISBN = {0-8218-0437-5},
   MRCLASS = {35J60 (35-01 35B45 35B65 35Dxx)},
  MRNUMBER = {1351007},
MRREVIEWER = {P.\ Lindqvist},
       DOI = {10.1090/coll/043},
       URL = {https://doi.org/10.1090/coll/043},
}

@article {Ishii92,
    AUTHOR = {Crandall, M. G. and Ishii, H. and Lions,
              P.-L.},
     TITLE = {User's guide to viscosity solutions of second order partial
              differential equations},
   JOURNAL = {Bull. Amer. Math. Soc. (N.S.)},
  FJOURNAL = {American Mathematical Society. Bulletin. New Series},
    VOLUME = {27},
      YEAR = {1992},
    NUMBER = {1},
     PAGES = {1--67},
      ISSN = {0273-0979,1088-9485},
   MRCLASS = {35J60 (35B05 35D05 35G20)},
  MRNUMBER = {1118699},
MRREVIEWER = {P.\ Szeptycki},
       DOI = {10.1090/S0273-0979-1992-00266-5},
       URL = {https://doi.org/10.1090/S0273-0979-1992-00266-5},
}

@article {Fusco90,
    AUTHOR = {Fusco, N. and Sbordone, C.},
     TITLE = {Local boundedness of minimizers in a limit case},
   JOURNAL = {Manuscripta Math.},
  FJOURNAL = {Manuscripta Mathematica},
    VOLUME = {69},
      YEAR = {1990},
    NUMBER = {1},
     PAGES = {19--25},
      ISSN = {0025-2611,1432-1785},
   MRCLASS = {49K30 (49K10)},
MRREVIEWER = {T.\ Zolezzi},
       DOI = {10.1007/BF02567909},
       URL = {https://doi.org/10.1007/BF02567909},
}

@article {Cupini15,
    AUTHOR = {Cupini, G. and Marcellini, P. and Mascolo, E.},
     TITLE = {Local boundedness of minimizers with limit growth conditions},
   JOURNAL = {J. Optim. Theory Appl.},
  FJOURNAL = {Journal of Optimization Theory and Applications},
    VOLUME = {166},
      YEAR = {2015},
    NUMBER = {1},
     PAGES = {1--22},
      ISSN = {0022-3239,1573-2878},
   MRCLASS = {49N60 (35J20 35J60 49J45)},
MRREVIEWER = {Anna\ Zatorska-Goldstein},
       DOI = {10.1007/s10957-015-0722-z},
       URL = {https://doi.org/10.1007/s10957-015-0722-z},
}

@article {Cupini17,
    AUTHOR = {Cupini, G. and Marcellini, P. and Mascolo, E.},
     TITLE = {Regularity of minimizers under limit growth conditions},
   JOURNAL = {Nonlinear Anal.},
  FJOURNAL = {Nonlinear Analysis. Theory, Methods \& Applications. An
              International Multidisciplinary Journal},
    VOLUME = {153},
      YEAR = {2017},
     PAGES = {294--310},
      ISSN = {0362-546X,1873-5215},
   MRCLASS = {49N60 (35J20 35J60)},
MRREVIEWER = {Shiah-Sen\ Wang},
       DOI = {10.1016/j.na.2016.06.002},
       URL = {https://doi.org/10.1016/j.na.2016.06.002},
}

@article {Dibenedetto16,
    AUTHOR = {Dibenedetto, E. and Gianazza, U. and Vespri, V.},
     TITLE = {Remarks on local boundedness and local {H}\"older continuity of local weak solutions to anisotropic {$p$}-{L}aplacian type equations},
   JOURNAL = {J. Elliptic Parabol. Equ.},
  FJOURNAL = {Journal of Elliptic and Parabolic Equations},
    VOLUME = {2},
      YEAR = {2016},
    NUMBER = {1-2},
     PAGES = {157--169},
      ISSN = {2296-9020,2296-9039},
   MRCLASS = {35J62 (35B45 35B65 35D30 35J70 35J92)},
  MRNUMBER = {3645942},
MRREVIEWER = {Teemu\ Lukkari},
       DOI = {10.1007/BF03377399},
       URL = {https://doi.org/10.1007/BF03377399},
}

@article {Kim252,
     url = {https://doi.org/10.1515/forum-2025-0318},
title = {Lipschitz regularity for fully nonlinear elliptic equations with (p, q)-growth},
author = {S.-S. Byun and H. Kim},
pages = {1555--1568},
volume = {38},
number = {5},
journal = {Forum Mathematicum},
doi = {doi:10.1515/forum-2025-0318},
year = {2026},
}

@article {Bousquet24,
    AUTHOR = {Bousquet, P. and Brasco, L. and Leone, C.},
     TITLE = {Singular orthotropic functionals with nonstandard growth
              conditions},
   JOURNAL = {Rev. Mat. Iberoam.},
  FJOURNAL = {Revista Matem\'atica Iberoamericana},
    VOLUME = {40},
      YEAR = {2024},
    NUMBER = {2},
     PAGES = {753--802},
      ISSN = {0213-2230,2235-0616},
   MRCLASS = {35J62 (35B65 49J10 49N60)},
       DOI = {10.4171/rmi/1446},
       URL = {https://doi.org/10.4171/rmi/1446},
}

@article {Fusco93,
    AUTHOR = {Fusco, N. and Sbordone, C.},
     TITLE = {Some remarks on the regularity of minima of anisotropic
              integrals},
   JOURNAL = {Comm. Partial Differential Equations},
  FJOURNAL = {Communications in Partial Differential Equations},
    VOLUME = {18},
      YEAR = {1993},
    NUMBER = {1-2},
     PAGES = {153--167},
      ISSN = {0360-5302,1532-4133},
   MRCLASS = {49N60},
MRREVIEWER = {J.\ S.\ Joel},
       DOI = {10.1080/03605309308820924},
       URL = {https://doi.org/10.1080/03605309308820924},
}

@article {Cianchi00,
    AUTHOR = {Cianchi, A.},
     TITLE = {Local boundedness of minimizers of anisotropic functionals},
   JOURNAL = {Ann. Inst. H. Poincar\'e{} C Anal. Non Lin\'eaire},
  FJOURNAL = {Annales de l'Institut Henri Poincar\'e{} C. Analyse Non
              Lin\'eaire},
    VOLUME = {17},
      YEAR = {2000},
    NUMBER = {2},
     PAGES = {147--168},
      ISSN = {0294-1449,1873-1430},
   MRCLASS = {49N60 (49J10)},
  MRNUMBER = {1753091},
MRREVIEWER = {Martin\ Fuchs},
       DOI = {10.1016/S0294-1449(99)00107-9},
       URL = {https://doi.org/10.1016/S0294-1449(99)00107-9},
}

@article {Cupini09,
    AUTHOR = {Cupini, G. and Marcellini, P. and Mascolo, E.},
     TITLE = {Regularity under sharp anisotropic general growth conditions},
   JOURNAL = {Discrete Contin. Dyn. Syst. Ser. B},
  FJOURNAL = {Discrete and Continuous Dynamical Systems. Series B. A Journal
              Bridging Mathematics and Sciences},
    VOLUME = {11},
      YEAR = {2009},
    NUMBER = {1},
     PAGES = {66--86},
      ISSN = {1531-3492,1553-524X},
   MRCLASS = {49N60 (35J20 35J60 35J70)},
MRREVIEWER = {Niko\ M.\ Marola},
       DOI = {10.3934/dcdsb.2009.11.67},
       URL = {https://doi.org/10.3934/dcdsb.2009.11.67},
}

@article {Liskevich09,
    AUTHOR = {Liskevich, V. and Skrypnik, I. I.},
     TITLE = {H\"older continuity of solutions to an anisotropic elliptic
              equation},
   JOURNAL = {Nonlinear Anal.},
  FJOURNAL = {Nonlinear Analysis. Theory, Methods \& Applications. An
              International Multidisciplinary Journal},
    VOLUME = {71},
      YEAR = {2009},
    NUMBER = {5-6},
     PAGES = {1699--1708},
      ISSN = {0362-546X,1873-5215},
   MRCLASS = {35J62 (35B65 35J70)},
  MRNUMBER = {2524385},
MRREVIEWER = {Jos\'e\ Carmona Tapia},
       DOI = {10.1016/j.na.2009.01.007},
       URL = {https://doi.org/10.1016/j.na.2009.01.007},
}

@article {Imbert16,
    AUTHOR = {Imbert, C. and Silvestre, L.},
     TITLE = {Estimates on elliptic equations that hold only where the
              gradient is large},
   JOURNAL = {J. Eur. Math. Soc. (JEMS)},
  FJOURNAL = {Journal of the European Mathematical Society (JEMS)},
    VOLUME = {18},
      YEAR = {2016},
    NUMBER = {6},
     PAGES = {1321--1338},
      ISSN = {1435-9855,1435-9863},
   MRCLASS = {35J60 (35B45 35B65 35D40 35J70)},
  MRNUMBER = {3500837},
MRREVIEWER = {Barbara\ Brandolini},
       DOI = {10.4171/JEMS/614},
       URL = {https://doi.org/10.4171/JEMS/614},
}

@article {Cabre97,
    AUTHOR = {Cabr\'e, X.},
     TITLE = {Nondivergent elliptic equations on manifolds with nonnegative
              curvature},
   JOURNAL = {Comm. Pure Appl. Math.},
  FJOURNAL = {Communications on Pure and Applied Mathematics},
    VOLUME = {50},
      YEAR = {1997},
    NUMBER = {7},
     PAGES = {623--665},
      ISSN = {0010-3640,1097-0312},
   MRCLASS = {58G03 (35J60 53C21)},
  MRNUMBER = {1447056},
MRREVIEWER = {David\ L.\ Finn},
       DOI = {10.1002/(SICI)1097-0312(199707)50:7<623::AID-CPA2>3.3.CO;2-B},
       URL =
              {https://doi.org/10.1002/(SICI)1097-0312(199707)50:7<623::AID-CPA2>3.3.CO;2-B},
}

@article {Savin07,
    AUTHOR = {Savin, O.},
     TITLE = {Small perturbation solutions for elliptic equations},
   JOURNAL = {Comm. Partial Differential Equations},
  FJOURNAL = {Communications in Partial Differential Equations},
    VOLUME = {32},
      YEAR = {2007},
    NUMBER = {4-6},
     PAGES = {557--578},
      ISSN = {0360-5302,1532-4133},
   MRCLASS = {35J60 (35B65)},
  MRNUMBER = {2334822},
MRREVIEWER = {Fabiana\ Leoni},
       DOI = {10.1080/03605300500394405},
       URL = {https://doi.org/10.1080/03605300500394405},
}

@article{Vedansh25,
      title={Harnack inequality for degenerate fully nonlinear parabolic equations}, 
      author={V. Arya and V. Julin},
      year={2025},
    JOURNAL = {arXiv preprint},
      pages={2506.10608},
      archivePrefix={arXiv},
      primaryClass={math.AP},
      url={https://arxiv.org/abs/2506.10608}, 
}

@article {Bousquet23,
    AUTHOR = {Bousquet, P. and Brasco, L. and Leone, C. and
              Verde, A.},
     TITLE = {Gradient estimates for an orthotropic nonlinear diffusion
              equation},
   JOURNAL = {Adv. Calc. Var.},
  FJOURNAL = {Advances in Calculus of Variations},
    VOLUME = {16},
      YEAR = {2023},
    NUMBER = {3},
     PAGES = {705--730},
      ISSN = {1864-8258,1864-8266},
   MRCLASS = {35K65 (35B65 35K92)},
  MRNUMBER = {4609806},
MRREVIEWER = {Feng\ Quan\ Li},
       DOI = {10.1515/acv-2021-0052},
       URL = {https://doi.org/10.1515/acv-2021-0052},
}

@article {Giaquinta87,
    AUTHOR = {Giaquinta, M.},
     TITLE = {Growth conditions and regularity, a counterexample},
   JOURNAL = {Manuscripta Math.},
  FJOURNAL = {Manuscripta Mathematica},
    VOLUME = {59},
      YEAR = {1987},
    NUMBER = {2},
     PAGES = {245--248},
      ISSN = {0025-2611,1432-1785},
   MRCLASS = {49B21 (49B22)},
  MRNUMBER = {905200},
MRREVIEWER = {Martin\ Brokate},
       DOI = {10.1007/BF01158049},
       URL = {https://doi.org/10.1007/BF01158049},
}

@article {Kim26,
      title={{ABP} estimate and {H}arnack inequality for a class of degenerate fully nonlinear pseudo-$p$-Laplacian equations}, 
      author={S.-S. Byun and H. Kim},
      year={2025},
      JOURNAL={arXiv preprint},
      pages={2509.24442},
      primaryClass={math.AP},
      url={https://arxiv.org/abs/2509.24442}, 
}

@article {DiBenedetto08,
    AUTHOR = {DiBenedetto, E. and Gianazza, U. and Vespri, V.},
     TITLE = {Harnack estimates for quasi-linear degenerate parabolic
              differential equations},
   JOURNAL = {Acta Math.},
  FJOURNAL = {Acta Mathematica},
    VOLUME = {200},
      YEAR = {2008},
    NUMBER = {2},
     PAGES = {181--209},
      ISSN = {0001-5962,1871-2509},
   MRCLASS = {35K55 (35B45 35B65)},
  MRNUMBER = {2413134},
MRREVIEWER = {Luca\ Lorenzi},
       DOI = {10.1007/s11511-008-0026-3},
       URL = {https://doi.org/10.1007/s11511-008-0026-3},
}

@article {Ciani25,
    AUTHOR = {Ciani, S. and Henriques, E. and Skrypnik, I. I.},
     TITLE = {On the continuity of solutions to anisotropic elliptic
              operators in the limiting case},
   JOURNAL = {Bull. Lond. Math. Soc.},
  FJOURNAL = {Bulletin of the London Mathematical Society},
    VOLUME = {57},
      YEAR = {2025},
    NUMBER = {5},
     PAGES = {1548--1567},
      ISSN = {0024-6093,1469-2120},
   MRCLASS = {35J70 (35B65 35D30 35J92)},
  MRNUMBER = {4913167},
       DOI = {10.1112/blms.70047},
       URL = {https://doi.org/10.1112/blms.70047},
}

@article {Marcellini89,
    AUTHOR = {Marcellini, P.},
     TITLE = {Regularity of minimizers of integrals of the calculus of
              variations with nonstandard growth conditions},
   JOURNAL = {Arch. Rational Mech. Anal.},
  FJOURNAL = {Archive for Rational Mechanics and Analysis},
    VOLUME = {105},
      YEAR = {1989},
    NUMBER = {3},
     PAGES = {267--284},
      ISSN = {0003-9527},
   MRCLASS = {49A50},
MRREVIEWER = {Yao\ Tian\ Shen},
       DOI = {10.1007/BF00251503},
       URL = {https://doi.org/10.1007/BF00251503},
}

@article {Marcellini91,
    AUTHOR = {Marcellini, P.},
     TITLE = {Regularity and existence of solutions of elliptic equations
              with {$p,q$}-growth conditions},
   JOURNAL = {J. Differential Equations},
  FJOURNAL = {Journal of Differential Equations},
    VOLUME = {90},
      YEAR = {1991},
    NUMBER = {1},
     PAGES = {1--30},
      ISSN = {0022-0396,1090-2732},
   MRCLASS = {35J15 (35D10)},
MRREVIEWER = {Philip\ W.\ Schaefer},
       DOI = {10.1016/0022-0396(91)90158-6},
       URL = {https://doi.org/10.1016/0022-0396(91)90158-6},
}

@article {Marcellini20,
    AUTHOR = {Marcellini, P.},
     TITLE = {Regularity under general and {$p,q$}-growth conditions},
   JOURNAL = {Discrete Contin. Dyn. Syst. Ser. S},
  FJOURNAL = {Discrete and Continuous Dynamical Systems. Series S},
    VOLUME = {13},
      YEAR = {2020},
    NUMBER = {7},
     PAGES = {2009--2031},
      ISSN = {1937-1632,1937-1179},
   MRCLASS = {35J62 (35B65 49J40 49N60)},
  MRNUMBER = {4097630},
MRREVIEWER = {Christopher\ Steven\ Goodrich},
       DOI = {10.3934/dcdss.2020155},
       URL = {https://doi.org/10.3934/dcdss.2020155},
}

@article {Mingione21,
    AUTHOR = {Mingione, G. and R\v adulescu, V.},
     TITLE = {Recent developments in problems with nonstandard growth and
              nonuniform ellipticity},
   JOURNAL = {J. Math. Anal. Appl.},
  FJOURNAL = {Journal of Mathematical Analysis and Applications},
    VOLUME = {501},
      YEAR = {2021},
    NUMBER = {1},
     PAGES = {Paper No. 125197, 41},
      ISSN = {0022-247X,1096-0813},
   MRCLASS = {49-02 (35J50 49J10 49N60)},
  MRNUMBER = {4258810},
MRREVIEWER = {Carlo\ Mariconda},
       DOI = {10.1016/j.jmaa.2021.125197},
       URL = {https://doi.org/10.1016/j.jmaa.2021.125197},
}

@article {Marcellini14,
    AUTHOR = {D\"uzg\"un, F. G. and Marcellini, P. and Vespri, V.},
     TITLE = {Space expansion for a solution of an anisotropic
              {$p$}-{L}aplacian equation by using a parabolic approach},
   JOURNAL = {Riv. Math. Univ. Parma (N.S.)},
  FJOURNAL = {Rivista di Matematica della Universit\`a{} di Parma. New
              Series. A Journal of Pure and Applied Mathematics},
    VOLUME = {5},
      YEAR = {2014},
    NUMBER = {1},
     PAGES = {93--111},
      ISSN = {0035-6298,2284-2578},
   MRCLASS = {35J62 (35B65 35J92)},
  MRNUMBER = {3289598},
}

@article {Ciani23,
    AUTHOR = {Ciani, S. and Mosconi, S. and Vespri, V.},
     TITLE = {Parabolic {H}arnack estimates for anisotropic slow diffusion},
   JOURNAL = {J. Anal. Math.},
  FJOURNAL = {Journal d'Analyse Math\'ematique},
    VOLUME = {149},
      YEAR = {2023},
    NUMBER = {2},
     PAGES = {611--642},
      ISSN = {0021-7670,1565-8538},
   MRCLASS = {35K55},
  MRNUMBER = {4594402},
MRREVIEWER = {Jeffrey\ R.\ Anderson},
       DOI = {10.1007/s11854-022-0261-0},
       URL = {https://doi.org/10.1007/s11854-022-0261-0},
}

@article {Piro-Vernier19,
    AUTHOR = {Piro-Vernier, S. and Ragnedda, F. and Vespri,
              V.},
     TITLE = {H\"older regularity for bounded solutions to a class of
              anisotropic operators},
   JOURNAL = {Manuscripta Math.},
  FJOURNAL = {Manuscripta Mathematica},
    VOLUME = {158},
      YEAR = {2019},
    NUMBER = {3-4},
     PAGES = {421--439},
      ISSN = {0025-2611,1432-1785},
   MRCLASS = {35J62 (35B45 35B65 35J70)},
  MRNUMBER = {3914957},
MRREVIEWER = {Francesco\ Della Pietra},
       DOI = {10.1007/s00229-018-1034-z},
       URL = {https://doi.org/10.1007/s00229-018-1034-z},
}

@article {Feo21,
    AUTHOR = {Feo, F. and V\'azquez, J. L. and Volzone, B.},
     TITLE = {Anisotropic {$p$}-{L}aplacian evolution of fast diffusion
              type},
   JOURNAL = {Adv. Nonlinear Stud.},
  FJOURNAL = {Advanced Nonlinear Studies},
    VOLUME = {21},
      YEAR = {2021},
    NUMBER = {3},
     PAGES = {523--555},
      ISSN = {1536-1365,2169-0375},
   MRCLASS = {35K55 (35A08 35B40 35K65)},
  MRNUMBER = {4294175},
       DOI = {10.1515/ans-2021-2136},
       URL = {https://doi.org/10.1515/ans-2021-2136},
}

@article {Liao20,
    AUTHOR = {Liao, N. and Skrypnik, I. I. and Vespri, V.},
     TITLE = {Local regularity for an anisotropic elliptic equation},
   JOURNAL = {Calc. Var. Partial Differential Equations},
  FJOURNAL = {Calculus of Variations and Partial Differential Equations},
    VOLUME = {59},
      YEAR = {2020},
    NUMBER = {4},
     PAGES = {Paper No. 116, 31},
      ISSN = {0944-2669,1432-0835},
   MRCLASS = {35J70 (35B65 35J92)},
  MRNUMBER = {4114267},
MRREVIEWER = {Armin\ Schikorra},
       DOI = {10.1007/s00526-020-01781-x},
       URL = {https://doi.org/10.1007/s00526-020-01781-x},
}

@article {Ciani232,
    AUTHOR = {Ciani, S. and Skrypnik, I. I. and Vespri, V.},
     TITLE = {On the local behavior of local weak solutions to some singular
              anisotropic elliptic equations},
   JOURNAL = {Adv. Nonlinear Anal.},
  FJOURNAL = {Advances in Nonlinear Analysis},
    VOLUME = {12},
      YEAR = {2023},
    NUMBER = {1},
     PAGES = {237--265},
      ISSN = {2191-9496,2191-950X},
   MRCLASS = {35J75 (35B65 35K92)},
  MRNUMBER = {4476932},
MRREVIEWER = {Haitao\ Wan},
       DOI = {10.1515/anona-2022-0275},
       URL = {https://doi.org/10.1515/anona-2022-0275},
}

@article {Baldelli24,
    AUTHOR = {Baldelli, L. and Ciani, S. and Skrypnik, I. and
              Vespri, V.},
     TITLE = {A note on the point-wise behaviour of bounded solutions for a
              non-standard elliptic operator},
   JOURNAL = {Discrete Contin. Dyn. Syst. Ser. S},
  FJOURNAL = {Discrete and Continuous Dynamical Systems. Series S},
    VOLUME = {17},
      YEAR = {2024},
    NUMBER = {5-6},
     PAGES = {1718--1732},
      ISSN = {1937-1632,1937-1179},
   MRCLASS = {35J75 (35B65 35K55 35K92 35R30)},
  MRNUMBER = {4762558},
       DOI = {10.3934/dcdss.2022143},
       URL = {https://doi.org/10.3934/dcdss.2022143},
}

@article {Byun25,
    AUTHOR = {Byun, S.-S. and Kim, H. and Oh, J.},
     TITLE = {Interior {$W^{2,\delta}$} type estimates for degenerate fully
              nonlinear elliptic equations with {$L^{n}$} data},
   JOURNAL = {J. Funct. Anal.},
  FJOURNAL = {Journal of Functional Analysis},
    VOLUME = {289},
      YEAR = {2025},
    NUMBER = {6},
     PAGES = {Paper No. 111007, 37},
      ISSN = {0022-1236,1096-0783},
   MRCLASS = {35B65 (35D40 35J60 35J70)},
  MRNUMBER = {4897643},
MRREVIEWER = {Driss\ Meskine},
       DOI = {10.1016/j.jfa.2025.111007},
       URL = {https://doi.org/10.1016/j.jfa.2025.111007},
}

@book {DiBenedetto12,
    AUTHOR = {DiBenedetto, E. and Gianazza, U. and Vespri, V.},
     TITLE = {Harnack's inequality for degenerate and singular parabolic
              equations},
    SERIES = {Springer Monographs in Mathematics},
 PUBLISHER = {Springer, New York},
      YEAR = {2012},
     PAGES = {xiv+278},
      ISBN = {978-1-4614-1583-1},
   MRCLASS = {35-02 (35B45 35B65 35K59 35K65 35K67)},
  MRNUMBER = {2865434},
MRREVIEWER = {Alain\ Brillard},
       DOI = {10.1007/978-1-4614-1584-8},
       URL = {https://doi.org/10.1007/978-1-4614-1584-8},
}

@article {Boccardo90,
    AUTHOR = {Boccardo, L. and Marcellini, P. and Sbordone, C.},
     TITLE = {{$L^\infty$}-regularity for variational problems with sharp
              nonstandard growth conditions},
   JOURNAL = {Boll. Un. Mat. Ital. A (7)},
  FJOURNAL = {Unione Matematica Italiana. Bollettino. A. Serie VII},
    VOLUME = {4},
      YEAR = {1990},
    NUMBER = {2},
     PAGES = {219--225},
   MRCLASS = {35B35 (49J45)},
  MRNUMBER = {1066774},
       DOI = {10.1007/bf01934372},
       URL = {https://doi.org/10.1007/bf01934372},
}

@article {Ochoa24,
    AUTHOR = {Ochoa, P. and Valverde, F. R.},
     TITLE = {Equivalence and regularity of weak and viscosity solutions for
              the anisotropic {${\bf p}(\cdot)$}-{L}aplacian},
   JOURNAL = {NoDEA Nonlinear Differential Equations Appl.},
  FJOURNAL = {NoDEA. Nonlinear Differential Equations and Applications},
    VOLUME = {31},
      YEAR = {2024},
    NUMBER = {5},
     PAGES = {Paper No. 88, 35},
      ISSN = {1021-9722,1420-9004},
   MRCLASS = {35J62 (35D30 35D40 35J92 46E30)},
  MRNUMBER = {4771102},
       DOI = {10.1007/s00030-024-00981-0},
       URL = {https://doi.org/10.1007/s00030-024-00981-0},
}

\end{document}